\documentclass[11pt]{article}
\usepackage{graphicx,subfigure,amsmath,amsfonts,bm,cite,epsfig,epsf,url,alg}
\usepackage{fullpage}
\usepackage[small,bf]{caption}
\usepackage{color}

\newtheorem{theorem}{Theorem}[section]
\newtheorem{lemma}[theorem]{Lemma}
\newtheorem{corollary}[theorem]{Corollary}

\newcommand{\R}{\mathbb{R}}

\newcommand{\<}{\langle}
\renewcommand{\>}{\rangle}

\newcommand{\goto}{\rightarrow}

\newcommand{\cA}{\mathcal{A}}
\newcommand{\cP}{\mathcal{P}}
\newcommand{\cD}{\mathcal{D}}
\newcommand{\cF}{\mathcal{F}}
\newcommand{\cL}{\mathcal{L}}

\newcommand{\vct}[1]{\bm{#1}}
\newcommand{\mtx}[1]{\bm{#1}}

\newcommand{\rank}{\operatorname{rank}}

\newcommand{\diag}{\operatorname{diag}}
\newcommand{\trace}{\operatorname{trace}}

\numberwithin{equation}{section}

\newcommand{\mymathbf}[1]{\mbox{\boldmath$#1$}}

\newenvironment{proof}{\noindent\emph{Proof.}}{\hfill\fbox{}\vspace*{1mm}}

\title{A Singular Value Thresholding Algorithm for Matrix Completion}

\author{Jian-Feng Cai$^{\dagger}$ ~~Emmanuel J. Cand\`es$^{\sharp}$
  ~~ Zuowei Shen$^{\S}$\\
  \vspace{-.1cm}\\
  $\dagger$ Temasek Laboratories, National University of
  Singapore, Singapore 117543\\
  \vspace{-.3cm}\\
  $\sharp$ Applied and Computational Mathematics,
  Caltech, Pasadena, CA 91125\\
  \vspace{-.3cm}\\
  $\S$ Department of Mathematics, National University of Singapore,
  Singapore 117543}

\date{September 2008}

\begin{document}
\maketitle

\begin{abstract}
  This paper introduces a novel algorithm to approximate the matrix
  with minimum nuclear norm among all matrices obeying a set of convex
  constraints.  This problem may be understood as the convex
  relaxation of a rank minimization problem, and arises in many
  important applications as in the task of recovering a large matrix
  from a small subset of its entries (the famous Netflix problem).
  Off-the-shelf algorithms such as interior point methods are not
  directly amenable to large problems of this kind with over a million
  unknown entries.

  This paper develops a simple first-order and easy-to-implement
  algorithm that is extremely efficient at addressing problems in
  which the optimal solution has low rank.  The algorithm is iterative
  and produces a sequence of matrices $\{\mtx{X}^k,\mtx{Y}^k\}$ and at
  each step, mainly performs a soft-thresholding operation on the
  singular values of the matrix $\mtx{Y}^k$. There are two remarkable
  features making this attractive for low-rank matrix completion
  problems. The first is that the soft-thresholding operation is
  applied to a sparse matrix; the second is that the rank of the
  iterates $\{\mtx{X}^k\}$ is empirically nondecreasing. Both these
  facts allow the algorithm to make use of very minimal storage space
  and keep the computational cost of each iteration low.  On the
  theoretical side, we provide a convergence analysis showing that the
  sequence of iterates converges. On the practical side, we provide
  numerical examples in which $1,000 \times 1,000$ matrices are
  recovered in less than a minute on a modest desktop computer. We
  also demonstrate that our approach is amenable to very large scale
  problems by recovering matrices of rank about 10 with nearly a
  billion unknowns from just about 0.4\% of their sampled entries. Our
  methods are connected with the recent literature on linearized
  Bregman iterations for $\ell_1$ minimization, and we develop a
  framework in which one can understand these algorithms in terms of
  well-known Lagrange multiplier algorithms.
\end{abstract}


\section{Introduction}

\subsection{Motivation}

There is a rapidly growing interest in the recovery of an unknown
low-rank or approximately low-rank matrix from very limited
information. This problem occurs in many areas of engineering and
applied science such as machine learning \cite{Abernethy06,
  Argyriou07, Amit07}, control \cite{Mesbahi97} and computer vision,
see \cite{Tomasi}. As a motivating example, consider the problem of
recovering a data matrix from a sampling of its entries. This
routinely comes up whenever one collects partially filled out surveys,
and one would like to infer the many missing entries. In the area of
recommender systems, users submit ratings on a subset of entries in a
database, and the vendor provides recommendations based on the user's
preferences. Because users only rate a few items, one would like to
infer their preference for unrated items; this is the famous Netflix
problem \cite{NetflixPrize}.  Recovering a rectangular matrix from a
sampling of its entries is known as the {\em matrix completion}
problem. The issue is of course that this problem is extraordinarily
ill posed since with fewer samples than entries, we have infinitely
many completions. Therefore, it is apparently impossible to identify
which of these candidate solutions is indeed the ``correct'' one
without some additional information.

In many instances, however, the matrix we wish to recover has low rank
or approximately low rank. For instance, the Netflix data matrix of
all user-ratings may be approximately low-rank because it is commonly
believed that only a few factors contribute to anyone's taste or
preference. In computer vision, inferring scene geometry and camera
motion from a sequence of images is a well-studied problem known as
the structure-from-motion problem. This is an ill-conditioned problem
for objects may be distant with respect to their size, or especially
for ``missing data'' which occur because of occlusion or tracking
failures. However, when properly stacked and indexed, these images
form a matrix which has very low rank (e.g.~rank 3 under orthography)
\cite{Tomasi,ChenSuter}.  Other examples of low-rank matrix fitting
abound; e.g.~in control (system identification), machine learning
(multi-class learning) and so on.  Having said this, the premise that
the unknown has (approximately) low  rank radically changes the problem,
making the search for solutions feasible since the lowest-rank
solution now tends to be the right one.

In a recent paper \cite{CR:XXX:08}, Cand\`es and Recht showed that
matrix completion is not as ill-posed as people thought. Indeed, they
proved that most low-rank matrices can be recovered {\em exactly} from
most sets of sampled entries even though these sets have surprisingly
small cardinality, and more importantly, they proved that this can
be done by solving a simple {\em convex} optimization problem. To
state their results, suppose to simplify that the unknown matrix
$\mtx{M} \in \R^{n \times n}$ is square, and that one has available
$m$ sampled entries $\{\mtx{M}_{ij} : (i, j) \in \Omega\}$ where
$\Omega$ is a random subset of cardinality $m$. Then \cite{CR:XXX:08}
proves that most matrices $\mtx{M}$ of rank $r$ can be perfectly
recovered by solving the optimization problem
\begin{equation}
  \label{eqn:min}
  \begin{array}{ll}
    \textrm{minimize}   & \quad \|\mtx{X}\|_*\\
    \textrm{subject to} & \quad X_{ij} = M_{ij}, \quad (i,j) \in \Omega,
 \end{array}
\end{equation}
provided that the number of samples obeys
\begin{equation}\label{conm}
m \ge C n^{6/5} r \log n
\end{equation}
for some positive numerical constant $C$.\footnote{Note that an $n
  \times n$ matrix of rank $r$ depends upon $r(2n-r)$ degrees of
  freedom.}  In \eqref{eqn:min}, the functional $\|\mtx{X}\|_*$ is the
nuclear norm of the matrix $\mtx{M}$, which is the sum of its singular
values. The optimization problem \eqref{eqn:min} is convex and can be
recast as a semidefinite program \cite{FazelThesis,fazelRank}. In some
sense, this is the tightest convex relaxation of the NP-hard rank
minimization problem
\begin{equation}
  \label{eq:rank}
  \begin{array}{ll}
    \textrm{minimize}   & \quad \text{rank}(\mtx{X})\\
    \textrm{subject to} & \quad X_{ij} = M_{ij}, \quad (i,j) \in \Omega,
 \end{array}
\end{equation}
since the nuclear ball $\{\mtx{X} : \|\mtx{X}\|_* \le 1\}$ is the
convex hull of the set of rank-one matrices with spectral norm bounded
by one. Another interpretation of Cand\`es and Recht's result is that
under suitable conditions, the rank minimization program
\eqref{eq:rank} and the convex program \eqref{eqn:min} are {\em
  formally equivalent} in the sense that they have exactly the same
unique solution.

\subsection{Algorithm outline}

Because minimizing the nuclear norm both provably recovers the
lowest-rank matrix subject to constraints (see \cite{Recht07} for
related results) and gives generally good empirical results in a
variety of situations, it is understandably of great interest to
develop numerical methods for solving \eqref{eqn:min}.  In
\cite{CR:XXX:08}, this optimization problem was solved using one of
the most advanced semidefinite programming solvers, namely, SDPT3
\cite{TTT:SDPT3}. This solver and others like
SeDuMi 
are based on interior-point methods, and are problematic when the size
of the matrix is large because they need to solve huge systems of
linear equations to compute the Newton direction. In fact, SDPT3 can
only handle $n \times n$ matrices with $n \le 100$. Presumably, one
could resort to iterative solvers such as the method of conjugate
gradients to solve for the Newton step but this is problematic as well
since it is well known that the condition number of the Newton system
increases rapidly as one gets closer to the solution. In addition,
none of these general purpose solvers use the fact that the solution
may have low rank. We refer the reader to \cite{VandenbergheNuc} for
some recent progress on interior-point methods concerning some special
nuclear norm-minimization problems.

This paper develops the {\em singular value thresholding} algorithm
for approximately solving the nuclear norm minimization problem
\eqref{eqn:min} and by extension, problems of the form
\begin{equation}
  \label{eqn:nuc_norm}
  \begin{array}{ll}
    \textrm{minimize}   & \quad \|\mtx{X}\|_*\\
    \textrm{subject to} & \quad {\cal A}(\mtx{X})  = \vct{b},
 \end{array}
\end{equation}
where ${\cal A}$ is a linear operator acting on the space of $n_1
\times n_2$ matrices and $\vct{b} \in \R^m$.  This algorithm is a
simple first-order method, and is especially well suited for problems
of very large sizes in which the solution has low rank.  We sketch
this algorithm in the special matrix completion setting and let
$\mathcal{P}_{\Omega}$ be the orthogonal projector onto the span of
matrices vanishing outside of $\Omega$ so that the $(i,j)$th component
of $\mathcal{P}_{\Omega}(\mtx{X})$ is equal to $X_{ij}$ if $(i,j) \in
\Omega$ and zero otherwise. Our problem may be expressed as
\begin{equation}
  \label{eqn:nuc_norm2}
  \begin{array}{ll}
    \textrm{minimize}   & \quad \|\mtx{X}\|_*\\
    \textrm{subject to} & \quad {\cal P}_\Omega(\mtx{X}) = {\cal P}_\Omega(\mtx{M}),
 \end{array}
\end{equation}
with optimization variable $\mtx{X} \in \R^{n_1 \times n_2}$. Fix
$\tau > 0$ and a sequence $\{\delta_k\}_{k \ge 1}$ of scalar step
sizes.  Then starting with $\mtx{Y}^0 = 0 \in \R^{n_1 \times n_2}$,
the algorithm inductively defines
\begin{equation}\label{eqn:iter0}
\begin{cases}
  \mtx{X}^{k}  = \text{shrink}(\mtx{Y}^{k-1}, \tau),\cr
  \mtx{Y}^{k}  = \mtx{Y}^{k-1} + \delta_{k}
  \mathcal{P}_{\Omega}(\mtx{M}-\mtx{X}^k)
\end{cases}
\end{equation}
until a stopping criterion is reached. In \eqref{eqn:iter0},
$\text{shrink}(\mtx{Y},\tau)$ is a nonlinear function which applies a
soft-thresholding rule at level $\tau$ to the singular values of the
input matrix, see Section \ref{sec:alg} for details.  The key property
here is that for large values of $\tau$, the sequence $\{\mtx{X}^k\}$
converges to a solution which very nearly minimizes
\eqref{eqn:nuc_norm2}.  Hence, at each step, one only needs to compute
at most one singular value decomposition and perform a few elementary
matrix additions. Two important remarks are in order:
\begin{enumerate}
\item {\em Sparsity.} For each $k \ge 0$, $\mtx{Y}^k$ vanishes outside
  of $\Omega$ and is, therefore, sparse, a fact which can be used to
  evaluate the shrink function rapidly.

\item {\em Low-rank property.} The matrices $\mtx{X}^k$ turn out to
  have low rank, and hence the algorithm has minimum storage
  requirement since we only need to keep principal factors in memory.
\end{enumerate}

Our numerical experiments demonstrate that the proposed algorithm can
solve problems, in Matlab, involving matrices of size $30,000 \times
30,000$ having close to a billion unknowns in 17 minutes on a standard
desktop computer with a 1.86 GHz CPU (dual core with Matlab's
multithreading option enabled) and 3 GB of memory.
As a consequence, the singular value thresholding algorithm may
become a rather powerful computational tool for large scale matrix
completion.

\subsection{General formulation}

The singular value thresholding algorithm can be adapted to deal with
other types of convex constraints. For instance, it may address
problems of the form
\begin{equation}
  \label{eqn:nuc_norm}
  \begin{array}{ll}
    \textrm{minimize}   & \quad \|\mtx{X}\|_*\\
    \textrm{subject to} & \quad f_i(\mtx{X}) \le 0, \quad i = 1, \ldots, m,
 \end{array}
\end{equation}
where each $f_i$ is a Lipschitz convex function (note that one can
handle linear equality constraints by considering pairs of affine
functionals). In the simpler case where the $f_i$'s are affine
functionals, the general algorithm goes through a sequence of
iterations which greatly resemble \eqref{eqn:iter0}. This is useful
because this enables the development of numerical algorithms which are
effective for recovering matrices from a small subset of sampled
entries possibly contaminated with noise.

\subsection{Contents and notations}

The rest of the paper is organized as follows. In Section
\ref{sec:alg}, we derive the singular value thresholding (SVT)
algorithm for the matrix completion problem, and recasts it in terms
of a well-known Lagrange multiplier algorithm. In Section
\ref{sec:general}, we extend the SVT algorithm and formulate a general
iteration which is applicable to general convex constraints. In
Section \ref{sec:conv}, we establish the convergence results for the
iterations given in Sections \ref{sec:alg} and \ref{sec:general}.
We demonstrate the performance and effectiveness of the algorithm
through numerical examples in Section \ref{sec:num}, and review
additional implementation details. Finally, we conclude the paper with
a short discussion in Section \ref{sec:discussion}.

Before continuing, we provide here a brief summary of the notations
used throughout the paper. Matrices are bold capital, vectors are bold
lowercase and scalars or entries are not bold. For instance, $\mtx{X}$
is a matrix and $X_{ij}$ its $(i,j)$th entry. Likewise, $\vct{x}$ is a
vector and $x_i$ its $i$th component. The nuclear norm of a matrix is
denoted by $\|\mtx{X}\|_*$, the Frobenius norm by $\|\mymathbf{X}\|_F$
and the spectral norm by $\|\mymathbf{X}\|_2$; note that these are
respectively the 1-norm, the 2-norm and the sup-norm of the vector of
singular values. The adjoint of a matrix $\mtx{X}$ is $\mtx{X}^*$ and
similarly for vectors. The notation $\diag(\mymathbf{x})$, where
$\vct{x}$ is a vector, stands for the diagonal matrix with $\{x_i\}$
as diagonal elements. We denote by $\langle\mtx{X}, \mtx{Y}\rangle =
\trace(\mtx{X}^*\mtx{Y})$ the standard inner product between two
matrices ($\|\mtx{X}\|_F^2 = \langle\mtx{X},\mtx{X}\rangle$).  The
Cauchy-Schwarz inequality gives
$\langle\mymathbf{X},\mymathbf{Y}\rangle\leq\|\mymathbf{X}\|_F\|\mymathbf{Y}\|_F$
and it is well known that we also have
$\langle\mymathbf{X},\mymathbf{Y}\rangle\leq\|\mymathbf{X}\|_*\|\mymathbf{Y}\|_2$
(the spectral and nuclear norms are dual from one another), see
e.g.~\cite{CR:XXX:08,Recht07}.

\section{The Singular Value Thresholding Algorithm }
\label{sec:alg}

This section introduces the singular value thresholding algorithm
and discusses some of its basic properties. We begin with the
definition of a key building block, namely, the singular value
thresholding operator.

\subsection{The singular value shrinkage operator}

Consider the singular value decomposition (SVD) of a matrix $\mtx{X}
\in \R^{n_1 \times n_2}$ of rank $r$
\begin{equation}
  \label{eq:svd}
  \mtx{X} = \mtx{U}  \mtx{\Sigma}   \mtx{V}^*,
\quad   \mtx{\Sigma}  = \diag(\{\sigma_i\}_{1 \le i \le r}),
\end{equation}
where $\mtx{U}$ and $\mtx{V}$ are respectively $n_1 \times r$ and $n_2
\times r$ matrices with orthonormal columns, and the singular values
$\sigma_i$ are positive (unless specified otherwise, we will always
assume that the SVD of a matrix is given in the reduced form
above). For each $\tau \ge 0$, we introduce the soft-thresholding
operator $\mathcal{D}_\tau$ defined as follows:
\begin{equation}
\label{eqn:DlamM2} \mathcal{D}_{\tau}(\mtx{X}):=  \mtx{U}
\mathcal{D}_{\tau}(\mtx{\Sigma}) \mtx{V}^*, \quad
\mathcal{D}_{\tau}(\mtx{\Sigma}) = \diag(\{\sigma_i - \tau)_+\}),
\end{equation}
where $t_+$ is the positive part of $t$, namely, $t_+ = \max(0,t)$. In
words, this operator simply applies a soft-thresholding rule to the
singular values of $\mtx{X}$, effectively shrinking these towards
zero. This is the reason why we will also refer to this transformation
as the {\em singular value shrinkage} operator. Even though the SVD
may not be unique, it is easy to see that the singular value shrinkage
operator is well defined and we do not elaborate further on this
issue.  In some sense, this shrinkage operator is a straightforward
extension of the soft-thresholding rule for scalars and vectors. In
particular, note that if many of the singular values of $\mtx{X}$ are
below the threshold $\tau$, the rank of $\mathcal{D}_{\tau}(\mtx{X})$
may be considerably lower than that of $\mtx{X}$, just like the
soft-thresholding rule applied to vectors leads to sparser outputs
whenever some entries of the input are below threshold.

The singular value thresholding operator is the proximity operator
associated with the nuclear norm. Details about the proximity
operator can be found in e.g.~\cite{HL:BOOK:93}.
\begin{theorem}\label{thm:prox}
  For each $\tau \ge 0$ and $\mtx{Y} \in \R^{n_1 \times n_2}$, the
  singular value shrinkage operator $\eqref{eqn:DlamM2}$ obeys
\begin{equation}
\label{eqn:DlamM}
\mathcal{D}_{\tau}(\mtx{Y}) = \arg\min_{\mtx{X}} \left\{
\frac12\|\mtx{X}-\mtx{Y}\|_F^2 + \tau\|\mymathbf{X}\|_{*}
\right\}.
\end{equation}
\end{theorem}
\begin{proof} Since the function $h_0(\mtx{X}) := \tau \|\mtx{X}\|_* +
  \frac{1}{2} \|\mtx{X}-\mtx{Y}\|_F^2$ is strictly convex, it is easy
  to see that there exists a unique minimizer, and we thus need to
  prove that it is equal to $\mathcal{D}_{\tau}(\mtx{Y})$.  To do
  this, recall the definition of a subgradient of a convex function $f
  : \R^{n_1 \times n_2} \goto \R$. We say that $\mtx{Z}$ is a
  subgradient of $f$ at $\mtx{X}_0$, denoted $\mtx{Z} \in \partial
  f(\mtx{X}_0)$, if
\begin{equation}
  \label{eq:subgradient}
  f(\mtx{X}) \ge f(\mtx{X}_0) + \<\mtx{Z}, \mtx{X} - \mtx{X}_0\>
\end{equation}
for all $\mtx{X}$.  Now $\hat{\mtx{X}}$ minimizes $h_0$ if and only if
$\mymathbf{0}$ is a subgradient of the functional $h_0$ at the point
$\hat{\mtx{X}}$, i.e.
\begin{equation}\label{eqn:subdiff}
\mymathbf{0} \in \hat{\mtx{X}}-\mtx{Y}+\tau\partial\|\hat{\mtx{X}}\|_*,
\end{equation}
where $\partial\|\hat{\mtx{X}}\|_*$ is the set of subgradients of the
nuclear norm. Let $\mtx{X} \in \R^{n_1 \times n_2}$ be an arbitrary
matrix and $\mtx{U} \mtx{\Sigma} \mtx{V}^*$ be its SVD.  It is known
\cite{CR:XXX:08,Lew:MP:03,Wat:LAA:92} that
\begin{equation}\label{eqn:subdiffNorm}
  \partial\|\mtx{X}\|_*=\left\{\mtx{U} \mtx{V}^* + \mtx{W} :
    ~\mtx{W}\in\mathbb{R}^{n_1 \times n_2},~~
    \mtx{U}^*\mtx{W}=0,~~
    \mtx{W} \mtx{V} =0,~~
    \|\mtx{W}\|_2\leq1\right\}.
\end{equation}

Set $\hat{\mtx{X}} := {\cal D}_\tau(\mtx{Y})$ for short. In order to
show that $\hat{\mtx{X}}$ obeys \eqref{eqn:subdiff}, decompose the SVD
of $\mtx{Y}$ as
\[
\mtx{Y} = \mtx{U}_0 \mtx{\Sigma}_0 \mtx{V}_0^* + \mtx{U}_1 \mtx{\Sigma}_1 \mtx{V}_1^*,
\]
where $\mtx{U}_0$, $\mtx{V}_0$ (resp.~$\mtx{U}_1$, $\mtx{V}_1$) are
the singular vectors associated with singular values greater than
$\tau$ (resp.~smaller than or equal to $\tau$). With these notations, we
have
\[
\hat{\mtx{X}} = \mtx{U}_0 (\mtx{\Sigma}_0 - \tau \mtx{I})\mtx{V}_0^*
\]
and, therefore,
\[
\mtx{Y} - \hat{\mtx{X}} = \tau (\mtx{U}_0 \mtx{V}_0^* +
\mtx{W}), \quad \mtx{W} = \tau^{-1} \mtx{U}_1 \mtx{\Sigma}_1
\mtx{V}_1^*.
\]
By definition, $\mtx{U}_0^* \mtx{W} = 0$, $\mtx{W} \mtx{V}_0 = 0$ and
since the diagonal elements of $\mtx{\Sigma}_1$ have magnitudes
bounded by $\tau$, we also have $\|\mtx{W}\|_2 \le 1$. Hence $\mtx{Y} -
\hat{\mtx{X}} \in \tau \partial\|\hat{\mtx{X}}\|_*$, which concludes the
proof.
\end{proof}

\subsection{Shrinkage iterations}

We are now in the position to introduce the singular value
thresholding algorithm. Fix $\tau > 0$ and a sequence $\{\delta_k\}$
of positive step sizes. Starting with $\mtx{Y}_0$, inductively define
for $k = 1, 2, \ldots$,
\begin{equation}\label{eqn:iter}
\begin{cases}
  \mtx{X}^{k}  = {\cal D}_\tau(\mtx{Y}^{k-1}),\cr
  \mtx{Y}^{k}  = \mtx{Y}^{k-1} + \delta_{k}
  \mathcal{P}_{\Omega}(\mtx{M}-\mtx{X}^k)
\end{cases}
\end{equation}
until a stopping criterion is reached (we postpone the discussion
this stopping criterion and of the choice of step sizes). This
shrinkage iteration is very simple to implement. At each step, we only
need to compute an SVD and perform elementary matrix operations. With
the help of a standard numerical linear algebra package, the whole
algorithm can be coded in just a few lines.

Before addressing further computational issues, we would like to make
explicit the relationship between this iteration and the original
problem \eqref{eqn:min}.  In Section \ref{sec:conv}, we will show that
the sequence $\{\mtx{X}^k\}$ converges to the unique solution of an
optimization problem closely related to \eqref{eqn:min}, namely,
\begin{equation}\label{eqn:minnuc+fro}
  \begin{array}{ll}
    \textrm{minimize}   & \quad \tau \|\mtx{X}\|_* +
\frac{1}{2} \|\mtx{X}\|_F^2\\
    \textrm{subject to} & \quad \mathcal{P}_\Omega(\mtx{X}) = \mathcal{P}_\Omega(\mtx{M}).
 \end{array}
\end{equation}
Furthermore, it is intuitive that the solution to this modified
problem converges to that of \eqref{eqn:nuc_norm2} as $\tau \to
\infty$ as shown in Section \ref{sec:general}. Thus by selecting a
large value of the parameter $\tau$, the sequence of iterates
converges to a matrix which nearly minimizes \eqref{eqn:min}.

As mentioned earlier, there are two crucial properties which make this
algorithm ideally suited for matrix completion.
\begin{itemize}
\item {\em Low-rank property.} A remarkable empirical fact is that the
  matrices in the sequence $\{\mtx{X}^k\}$ have low rank (provided, of
  course, that the solution to \eqref{eqn:minnuc+fro} has low
  rank). We use the word ``empirical'' because all of our numerical
  experiments have produced low-rank sequences but we cannot
  rigorously prove that this is true in general.  The reason for this
  phenomenon is, however, simple: because we are interested in large
  values of $\tau$ (as to better approximate the solution to
  \eqref{eqn:min}), the thresholding step happens to `kill' most of
  the small singular values and produces a low-rank output.
  In fact, our numerical results show that the rank of $\mtx{X}^{k}$
  is nondecreasing with $k$, and the maximum rank is reached in the
  last steps of the algorithm, see Section \ref{sec:num}.

  Thus, when the rank of the solution is substantially smaller than
  either dimension of the matrix, the storage requirement is low since
  we could store each $\mtx{X^k}$ in its SVD form (note that we only
  need to keep the current iterate and may discard earlier
  values).

\item {\em Sparsity.}  Another important property of the SVT algorithm
  is that the iteration matrix $\mtx{Y}^k$ is sparse. Since
  $\mtx{Y}^0=\mtx{0}$, we have by induction that $\mtx{Y}^{k}$
  vanishes outside of $\Omega$. The fewer entries available, the
  sparser $\mtx{Y}^k$. Because the sparsity pattern $\Omega$ is fixed
  throughout, one can then apply sparse matrix techniques to save
  storage. Also, if $|\Omega| = m$, the computational cost of updating
  $\mtx{Y}^k$ is of order $m$. Moreover, we can call
  subroutines supporting sparse matrix computations, which can further
  reduce computational costs.

  One such subroutine is the SVD. However, note that we do not need to
  compute the entire SVD of $\mtx{Y}^k$ to apply the singular value
  thresholding operator. Only the part corresponding to singular
  values greater than $\tau$ is needed. Hence, a good strategy is to
  apply the iterative Lanczos algorithm to compute the first few
  singular values and singular vectors. Because $\mtx{Y}^k$ is sparse,
  $\mtx{Y}^k$ can be applied to arbitrary vectors rapidly, and this
  procedure offers a considerable speedup over naive methods.
\end{itemize}

\subsection{Relation with other works}

Our algorithm is inspired by recent work in the area of $\ell_1$
minimization, and especially by the work on linearized Bregman
iterations for compressed sensing, see
\cite{COS:XXX:08:3,COS:XXX:08:2,COS:XXX:08,DO:XXX:07,YOGD:SIIMS:08,ODY:XXX:08}
for linearized Bregman iterations and
\cite{CR:IP:07,CRT:TIT:06,CT:TIT:05,CT:TIT:06,Don:TIT:06} for some
information about the field of compressed sensing. In this line of
work, linearized Bregman iterations are used to find the solution to
an underdetermined system of linear equations with minimum $\ell_1$
norm.  In fact, Theorem \ref{thm:prox} asserts that the singular value
thresholding algorithm can be formulated as a linearized Bregman
iteration.  Bregman iterations were first introduced in
\cite{OBGXY:MMS:05} as a convenient tool for solving computational
problems in the imaging sciences, and a later paper
\cite{YOGD:SIIMS:08} showed that they were useful for solving
$\ell_1$-norm minimization problems in the area of compressed sensing.
Linearized Bregman iterations were proposed in \cite{DO:XXX:07} to
improve performance of plain Bregman iterations, see also
\cite{YOGD:SIIMS:08}. Additional details together with a technique for
improving the speed of convergence called {\em kicking} are described
in \cite{ODY:XXX:08}. On the practical side, the paper
\cite{COS:XXX:08:2} applied Bregman iterations to solve a deblurring
problem while on the theoretical side, the references
\cite{COS:XXX:08:3,COS:XXX:08} gave a rigorous analysis of the
convergence of such iterations. New developments keep on coming out at
a rapid pace and recently, \cite{GO:XXX:08} introduced a new
iteration, the {\em split Bregman iteration}, to extend Bregman-type
iterations (such as linearized Bregman iterations) to problems
involving the minimization of $\ell_1$-like functionals such as
total-variation norms, Besov norms, and so forth.

When applied to $\ell_1$-minimization problems, linearized Bregman
iterations are sequences of soft-thresholding rules operating on
vectors. Iterative soft-thresholding algorithms in connection with
$\ell_1$ or total-variation minimization have quite a bit of history
in signal and image processing and we would like to mention the
works \cite{TVSynthesis,Lintner} for total-variation minimization,
\cite{Nowak_EM,DDD:CPAM:04,DTV:IPI:07} for $\ell_1$ minimization,
and
\cite{CCS:ACHA:08,CCSS:SISC:08,CS:NM:07,CCSS:SISC:03,ESQD:ACHA:05,FSM:CJ:07,SDC:AA:03,BBAC:ECCV:04}
for some recent applications in the area of image inpainting and
image restoration. Just as iterative soft-thresholding methods are
designed to find sparse solutions, our iterative singular value
thresholding scheme is designed to find a sparse vector of singular
values. In classical problems arising in the areas of compressed
sensing, and signal or image processing, the sparsity is expressed
in a known transformed domain and soft-thresholding is applied to
transformed coefficients. In contrast, the shrinkage operator
$\mathcal{D}_{\tau}$ is adaptive. The SVT not only discovers a
sparse singular vector but also the bases in which we have a sparse
representation.  In this sense, the SVT algorithm is an extension of
earlier iterative soft-thresholding schemes.

Finally, we would like to contrast the SVT iteration \eqref{eqn:iter}
with the popular iterative soft-thresholding algorithm used in many
papers in imaging processing and perhaps best known under the name of
Proximal Forward-Backward Splitting method (PFBS), see
\cite{CW:MMS:05,Nowak_EM,DDD:CPAM:04,YinFPC,CCS:ACHA:08} for example.
The constrained minimization problem \eqref{eqn:nuc_norm2} may be
relaxed into
\begin{equation}
\label{eqn:min_uncon}
 \textrm{minimize} \quad \lambda \|\mtx{X}\|_* +
 \frac{1}{2} \|\cP_{\Omega}(\mtx{X})-\cP_{\Omega}(\mtx{M})\|_F^2
\end{equation}
for some $\lambda > 0$. Theorem \ref{thm:prox} asserts that
$\cD_{\lambda}$ is the proximity operator of $\lambda \|\mtx{X}\|_*$
and Proposition 3.1(iii) in \cite{CW:MMS:05} gives that the solution
to this unconstrained problem is characterized by the fixed point
equation $\mtx{X} = \cD_{\lambda \delta}(\mtx{X} + \delta
P_{\Omega}(\mtx{M}-\mtx{X}))$ for each $\delta > 0$. One can then
apply a simplified version of the PFBS method (see (3.6) in
\cite{CW:MMS:05}) to obtain iterations of the form
\[
 \mtx{X}^{k} =
\cD_{\lambda \delta_{k-1}}(\mtx{X}^{k-1}+\delta_{k-1} P_{\Omega}(\mtx{M}-\mtx{X}^{k-1})).
\]
Introducing an intermediate matrix $\mtx{Y}^{k}$, this algorithm may
be expressed as
\begin{equation}
\label{eqn:iter_thresh}
\begin{cases}
 \mtx{X}^k = \cD_{\lambda \delta_{k-1}}(\mtx{Y}^{k-1}),\cr
 \mtx{Y}^{k} = \mtx{X}^{k}+\delta_kP_{\Omega}(\mtx{M}-\mtx{X}^{k}).
\end{cases}
\end{equation}
The difference with \eqref{eqn:iter} may seem subtle at
first---replacing $\mtx{X}^{k}$ in \eqref{eqn:iter_thresh} with
$\mtx{Y}^{k-1}$ and setting $\delta_k = \delta$ gives \eqref{eqn:iter}
with $\tau = \lambda \delta$---but has enormous consequences as this
gives entirely different algorithms.  First, they have different
limits: while \eqref{eqn:iter} converges to the solution of the
constrained minimization \eqref{eqn:minnuc+fro},
\eqref{eqn:iter_thresh} converges to the solution of
\eqref{eqn:min_uncon} provided that the sequence of step sizes is
appropriately selected.  Second, selecting a large $\lambda$ (or a
large value of $\tau = \lambda \delta$) in \eqref{eqn:iter_thresh}
gives a low-rank sequence of iterates and a limit with small nuclear
norm. The limit, however, does not fit the data and this is why one
has to choose a small or moderate value of $\lambda$ (or of $\tau =
\lambda \delta$).  However, when $\lambda$ is not sufficiently large,
the $\mtx{X}^k$ may not have low rank even though the solution has low
rank (and one may need to compute many singular vectors), and
$\mtx{Y}^k$ is not sufficiently sparse to make the algorithm
computationally attractive. Moreover, the limit does not necessary
have a small nuclear norm.  These are reasons why
\eqref{eqn:iter_thresh} is not suitable for matrix completion.



\subsection{Interpretation as a Lagrange multiplier method}
\label{sec:uzawa}

In this section, we recast the SVT algorithm as a type of Lagrange
multiplier algorithm known as Uzawa's algorithm. An important
consequence is that this will allow us to extend the SVT algorithm to
other problems involving the minimization of the nuclear norm under
convex constraints, see Section \ref{sec:general}. Further, another
contribution of this paper is that this framework actually recasts
linear Bregman iterations as a very special form of Uzawa's algorithm,
hence providing fresh and clear insights about
these 
iterations.

In what follows, we set $f_\tau(\mtx{X}) = \tau \|\mtx{X}\|_* +
\frac{1}{2} \|\mtx{X}\|_F^2$ for some fixed $\tau > 0$ and recall that
we wish to solve \eqref{eqn:minnuc+fro}
\[
  \begin{array}{ll}
    \textrm{minimize}   & \quad f_\tau(\mtx{X})\\
    \textrm{subject to} & \quad \mathcal{P}_\Omega(\mtx{X}) = \mathcal{P}_\Omega(\mtx{M}).
 \end{array}
\]
The Lagrangian for this problem is given by
\[
{\cal L}(\mtx{X},\mtx{Y}) = f_\tau(\mtx{X}) + \<\mtx{Y}, {\cal
  P}_\Omega(\mtx{M} - \mtx{X})\>,
\]
where $\mtx{Y} \in \R^{n_1 \times n_2}$.  Strong duality holds and
$\mtx{X}^\star$ and $\mtx{Y}^\star$ are primal-dual optimal if
$(\mtx{X}^\star, \mtx{Y}^\star)$ is a saddlepoint of the Lagrangian
${\cal L}(\mtx{X},\mtx{Y})$, i.e.~a pair obeying
\begin{equation}
\label{eq:saddlepoint}
\sup_{\mtx{Y}} \inf_{\mtx{X}} {\cal L}(\mtx{X}, \mtx{Y}) = {\cal L}(\mtx{X}^\star, \mtx{Y}^\star) = \inf_{\mtx{X}} \sup_{\mtx{Y}}
{\cal L}(\mtx{X}, \mtx{Y}).
\end{equation}
(The function $g_0(\mtx{Y}) = \inf_{\mtx{X}} {\cal L}(\mtx{X},
\mtx{Y})$ is called the dual function.) Uzawa's algorithm approaches the
problem of finding a saddlepoint with an iterative procedure.  From
$\mtx{Y}_0 = \mtx{0}$, say, inductively define
\begin{equation}
\label{eq:Lag1}
\begin{cases}
{\cal L}(\mtx{X}^{k}, \mtx{Y}^{k-1})
 = \min_{\mtx{X}} {\cal L}(\mtx{X},\mtx{Y}^{k-1})\\  \mtx{Y}^k
 = \mtx{Y}^{k-1} + \delta_k \mathcal{P}_\Omega(\mtx{M} -\mtx{X}^k),
\end{cases}
\end{equation}
where $\{\delta_k\}_{k \ge 1}$ is a sequence of positive step
sizes. Uzawa's algorithm is, in fact, a subgradient method applied to
the dual problem, where each step moves the current iterate in the
direction of the gradient or of a subgradient. Indeed, observe that
\begin{equation}\label{eqn:partialg0}
\partial_{\mtx{Y}} g_0(\mtx{Y}) = \partial_{\mtx{Y}} {\cal
  L}(\tilde{\mtx{X}}, \mtx{Y}) =
\mathcal{P}_\Omega(\mtx{M}-\tilde{\mtx{X}}),
\end{equation}
where $\tilde{\mtx{X}}$ is the minimizer of the Lagrangian for that
value of $\mtx{Y}$ so that a gradient descent update for $\mtx{Y}$ is
of the form
\[
\mtx{Y}^k = \mtx{Y}^{k-1} + \delta_k \partial_{\mtx{Y}}
g_0(\mtx{Y}^{k-1}) = \mtx{Y}^{k-1} + \delta_k
\mathcal{P}_\Omega(\mtx{M}-\mtx{X}^k).
\]

It remains to compute the minimizer of the Lagrangian \eqref{eq:Lag1},
and note that
\begin{equation}\label{eqn:argminequiv}
\arg \min \, f_\tau(\mtx{X}) + \<\mtx{Y}, {\cal P}_\Omega(\mtx{M} -
\mtx{X})\> = \arg \min \, \tau \|\mtx{X}\|_* + \frac{1}{2} \|\mtx{X} -
\mathcal{P}_\Omega \mtx{Y}\|^2_F.
\end{equation}
However, we know that the minimizer is given by
$\mathcal{D}_\tau(\mathcal{P}_\Omega(\mtx{Y}))$ and since $\mtx{Y}^{k}
= \mathcal{P}_\Omega(\mtx{Y}^k)$ for all $k \ge 0$, Uzawa's algorithm
takes the form
\begin{align*}
\begin{cases}
  \mtx{X}^k   = {\cal D}_{\tau} (\mtx{Y}^{k-1})\cr
  \mtx{Y}^k  = \mtx{Y}^{k-1} + \delta_k \mathcal{P}_\Omega(\mtx{M} - \mtx{X}^k),
\end{cases}
\end{align*}
which is exactly the update \eqref{eqn:iter}. This point of view
brings to bear many different mathematical tools for proving the
convergence of the singular value thresholding iterations. For an early
use of Uzawa's algorithm minimizing an $\ell_1$-like functional, the
total-variation norm, under linear inequality constraints, see
\cite{TVSynthesis}.

\section{General Formulation}
\label{sec:general}

This section presents a general formulation of the SVT algorithm for
approximately minimizing the nuclear norm of a matrix under convex
constraints.

\subsection{Linear equality constraints}

Set the objective functional $f_\tau(\mtx{X}) = \tau \|\mtx{X}\|_* +
\frac{1}{2}\|\mtx{X}\|_F^2$ for some fixed $\tau > 0$, and consider
the following optimization problem:
\begin{equation}
\label{eq:linear}
  \begin{array}{ll}
    \textrm{minimize}   & \quad f_\tau(\mtx{X})\\
    \textrm{subject to} & \quad \mathcal{A}(\mtx{X})  = \vct{b},
 \end{array}
\end{equation}
where $\mathcal{A}$ is a linear transformation mapping $n_1 \times
n_2$ matrices into $\R^m$ ($\mathcal{A}^*$ is the adjoint of
$\mathcal{A}$). This more general formulation is considered in
\cite{CR:XXX:08} and \cite{Recht07} as an extension of the matrix
completion problem.  Then the Lagrangian for this problem is of the
form
\begin{equation}
  \label{eq:Lagrangian1}
  {\cal L}(\mtx{X}, \vct{y}) = f_\tau(\mtx{X}) + \< \vct{y}, \vct{b} -
  \mathcal{A}(\mtx{X})\>,
\end{equation}
where $\mtx{X} \in \R^{n_1 \times n_2}$ and $\vct{y} \in \R^m$, and
starting with $\vct{y}^0 = \vct{0}$, Uzawa's iteration is given by
\begin{equation}
\label{eqn:itergeneral}
\begin{cases}
  \mymathbf{X}^{k} =
  \mathcal{D}_{\tau}(\mathcal{A}^*(\vct{y}^{k-1})),\cr \vct{y}^{k} =
  \vct{y}^{k-1} + \delta_k (\vct{b} - \mathcal{A}(\mtx{X}^k)).
\end{cases}
\end{equation}
The iteration \eqref{eqn:itergeneral} is of course the same as
\eqref{eqn:iter} in the case where $\mathcal{A}$ is a sampling
operator extracting $m$ entries with indices in $\Omega$ out of an
$n_1 \times n_2$ matrix. To verify this claim, observe that in this
situation, $\mathcal{A}^* \mathcal{A} = \mathcal{P}_\Omega$, and let
$\mtx{M}$ be any matrix obeying $\mathcal{A}(\mtx{M}) = \vct{b}$. Then
defining $\mtx{Y}^k = \mathcal{A}^*(\vct{y}^{k})$ and substituting
this expression in \eqref{eqn:itergeneral} gives \eqref{eqn:iter}.

\subsection{General convex constraints}

One can also adapt the algorithm to handle general convex
constraints. Suppose we wish to minimize $f_\tau(\mtx{X})$ defined as
before over a convex set $\mtx{X} \in \mathcal{C}$. To simplify, we
will assume that this convex set is given by
\[
\mathcal{C} = \{\mtx{X} : f_i(\mtx{X}) \le 0, \, \forall i = 1,
\ldots, m\},
\]
where the $f_i$'s are convex functionals (note that one can handle
linear equality constraints by considering pairs of affine
functionals).  The problem of interest is then of the form
\begin{equation}
\label{eq:convex}
  \begin{array}{ll}
    \textrm{minimize}   & \quad f_\tau(\mtx{X})\\
    \textrm{subject to} & \quad f_i(\mtx{X}) \le 0, \quad i = 1, \ldots, m.
 \end{array}
\end{equation}
Just as before, it is intuitive that as $\tau \to \infty$, the
solution to this problem converges to a minimizer of the nuclear norm
under the same constraints \eqref{eqn:nuc_norm} as shown in Theorem
\ref{thm:largemu2} at the end of this section.

Put $\mathcal{F}(\mtx{X}) := (f_1(\mtx{X}), \ldots, f_m(\mtx{X}))$
for short.  Then the Lagrangian for \eqref{eq:convex} is equal to
\[
{\cal L}(\mtx{X}, \vct{y}) = f_\tau(\mtx{X}) + \< \vct{y},
\mathcal{F}(\mtx{X})\>,
\]
where $\mtx{X} \in \R^{n_1 \times n_2}$ and $\vct{y} \in \R^m$ is now
a vector with nonnegative components denoted, as usual, by $\vct{y}
\ge \vct{0}$. One can apply Uzawa's method just as before with the
only modification that we will use a subgradient method with
projection to maximize the dual function since we need to make sure
that the successive updates $\vct{y}^k$ belong to the nonnegative
orthant. This gives
\begin{equation}
\label{eqn:itergeneral2}
\begin{cases}
  \mymathbf{X}^{k} = \arg \min \, \{f_\tau(\mtx{X}) + \<\vct{y}^{k-1},  \mathcal{F}(\mtx{X})\>\}, \cr
  \vct{y}^{k} =
  [\vct{y}^{k-1} + \delta_k \mathcal{F}(\mtx{X}^k) ]_+.
\end{cases}
\end{equation}
Above, $\vct{x}_+$ is of course the vector with entries equal to
$\max(x_i,0)$.  When $\mathcal{F}$ is an affine mapping of the form
$\vct{b} - \cA(\mtx{X})$ so that one solves
\[
  \begin{array}{ll}
    \textrm{minimize}   & \quad f_\tau(\mtx{X})\\
    \textrm{subject to} & \quad \cA(\mtx{X}) \ge \vct{b},
 \end{array}
\]
 this simplifies to
\begin{equation}
\label{eqn:itergeneral3}
\begin{cases}
  \mymathbf{X}^{k} = \mathcal{D}_\tau(\mathcal{A}^*(\vct{y}^{k-1})), \cr
  \vct{y}^{k} =
  [\vct{y}^{k-1} + \delta_k (\vct{b} - \mathcal{A}(\mtx{X}^k))]_+,
\end{cases}
\end{equation}
and thus the extension to linear inequality constraints is
straightforward.

\subsection{Example}
\label{sec:Dantzig}

An interesting example concerns the extension of the Dantzig selector
\cite{DS} to matrix problems. Suppose we have available linear
measurements about a matrix $\mtx{M}$ of interest
\begin{equation}
  \label{eq:noisy}
  \vct{b} = \mathcal{A}(\mtx{M}) + \vct{z},
\end{equation}
where $\vct{z} \in \R^m$ is a noise vector. Then under these
circumstances, one might want to find the matrix which minimizes the
nuclear norm among all matrices which are consistent with the data
$\vct{b}$. Inspired by the work on the Dantzig selector which was
originally developed for estimating sparse parameter vectors from
noisy data, one could approach this problem by solving
\begin{equation}
\label{eq:DS}
  \begin{array}{ll}
    \textrm{minimize}   & \quad \|\mtx{X}\|_*\\
    \textrm{subject to} & \quad |\text{\bf vec}(\mathcal{A}^*(\vct{r}))| \le
    \text{\bf vec}(\mtx{E}), \quad \vct{r} := \vct{b} - \mathcal{A}(\mtx{X}),
 \end{array}
\end{equation}
where $\mtx{E}$ is an array of tolerances, which is adjusted to fit
the noise statistics \cite{DS}. Above, $\text{\bf vec}(\mtx{A}) \le
\text{\bf vec}(\mtx{B})$, for any two matrices $\mtx{A}$ and
$\mtx{B}$, means componentwise inequalities; that is, $A_{ij} \le
B_{ij}$ for all indices $i, j$. We use this notation as not to confuse
the reader with the positive semidefinite ordering.  In the case of
the matrix completion problem where $\mathcal{A}$ extracts sampled
entries indexed by $\Omega$, one can always see the data vector as the
sampled entries of some matrix $\mtx{B}$ obeying $\cP_\Omega(\mtx{B})
= \cA^*(\vct{b})$; the constraint is then natural for it may be expressed
as
\[
|B_{ij} - X_{ij}| \le E_{ij}, \quad (i,j) \in \Omega,
\]
If $\vct{z}$ is white noise with standard deviation $\sigma$, one may
want to use a multiple of $\sigma$ for $E_{ij}$.  In words, we are
looking for a matrix with minimum nuclear norm under the constraint
that all of its sampled entries do not deviate too much from what has
been observed.

Let $\mtx{Y}_+ \in \R^{n_1 \times n_2}$ (resp.~$\mtx{Y}_- \in \R^{n_1
  \times n_2}$) be the Lagrange multiplier associated with the
componentwise linear inequality constraints $\text{\bf vec}(\mathcal{A}^*(\vct{r})) \le
\text{\bf vec}(\mtx{E})$ (resp.~$-\text{\bf vec}(\mathcal{A}^*(\vct{r})) \le \text{\bf vec}(\mtx{E}$)).  Then starting
with $\mtx{Y}_{\pm}^0 = \mtx{0}$, the SVT iteration for this problem
is of the form
\begin{equation}
\label{eqn:iterDS}
\begin{cases}
  \mymathbf{X}^{k} =
  \mathcal{D}_\tau(\mathcal{A}^*\mathcal{A}(\mtx{Y}_+^{k-1}-\mtx{Y}_{-}^{k-1})),
  \cr \mtx{Y}_{\pm}^k = [\mtx{Y}_{\pm}^{k-1} + \delta_k(\pm
  \mathcal{A}^*(\vct{r}^k) - \mtx{E})]_+,\quad \vct{r}^k = \vct{b}^k -
  \mathcal{A}(\mtx{X}^k),
\end{cases}
\end{equation}
where again $[\cdot]_+$ is applied componentwise.

We conclude by noting that in the matrix completion problem where
$\mathcal{A}^*\mathcal{A} = \mathcal{P}_\Omega$ and one observes $
\mathcal{P}_\Omega (\mtx{B})$, one can check that this iteration
simplifies to
\begin{equation}
\label{eqn:iterDS2}
\begin{cases}
  \mymathbf{X}^{k} =
  \mathcal{D}_\tau(\mtx{Y}_+^{k-1}-\mtx{Y}_{-}^{k-1}), \cr
  \mtx{Y}_{\pm}^k = [\mtx{Y}_{\pm}^{k-1} + \delta_k
  \mathcal{P}_\Omega(\pm (\mtx{B}-\mtx{X}^k) - \mtx{E})]_+.
\end{cases}
\end{equation}
Again, this is easy to implement and whenever the solution has low
rank, the iterates $\mtx{X}^k$ have low rank as well.

\subsection{When the proximal problem gets close}

We now show that minimizing the proximal objective $f_\tau(\mtx{X}) =
\tau \|\mtx{X}\|_* + \frac{1}{2} \|\mtx{X}\|_F^2$ is the same as
minimizing the nuclear norm in the limit of large $\tau$'s. The
theorem below is general and covers the special case of linear
equality constraints as in \eqref{eqn:minnuc+fro}.

\begin{theorem}\label{thm:largemu2}
  Let $\mtx{X}_{\tau}^{\star}$ be the solution to \eqref{eq:convex}
  and $\mtx{X}_\infty$ be the minimum Frobenius-norm solution to
  \eqref{eqn:nuc_norm} defined as
  \begin{equation}\label{eqn:minfro2}
    \mtx{X}_\infty:= \arg \min_{\mtx{X}}\{\|\mtx{X}\|_F^2~:~\mtx{X}\text{ is a solution of \eqref{eqn:nuc_norm}}\}.
  \end{equation}
  Assume that the $f_i(\mtx{X})$'s, $1 \le i \le m$, are convex and
  lower semi-continuous. Then
\begin{equation}\label{eqn:limitXmustar2}
  \lim_{\tau\to\infty}\|\mtx{X}_{\tau}^{\star}-\mtx{X}_\infty\|_F=0.
\end{equation}
\end{theorem}
\begin{proof}
It follows from the definition of $\mtx{X}_{\tau}^{\star}$ and
$\mtx{X}_\infty$ that
\begin{equation}\label{eqn:XmustarX0}
  \|\mtx{X}_{\tau}^{\star}\|_*+\frac{1}{2\tau}\|\mtx{X}_{\tau}^{\star}\|_F^2\leq
  \|\mtx{X}_{\infty}\|_*+\frac{1}{2\tau}\|\mtx{X}_{\infty}\|_F^2,\quad \text{ and } \quad
  \|\mtx{X}_{\infty}\|_*\leq\|\mtx{X}_{\tau}^{\star}\|_*.
\end{equation}
Summing these two inequalities gives
\begin{equation}\label{eqn:boundedXmustar}
\|\mtx{X}_{\tau}^{\star}\|_F^2\leq\|\mtx{X}_{\infty}\|_F^2,
\end{equation}
which implies that $\|\mtx{X}_{\tau}^{\star}\|_F^2$ is bounded
uniformly in $\tau$.  Thus, we would prove the theorem if we could
establish that any convergent subsequence
$\{\mtx{X}^{\star}_{\tau_k}\}_{k \ge 1}$ must converge to
$\mtx{X}_\infty$.

Consider an arbitrary converging subsequence
$\{\mtx{X}^{\star}_{\tau_k}\}$ and set $\mtx{X}_c := \lim_{k \goto
  \infty} \mtx{X}^{\star}_{\tau_k}$. Since for each $1 \le i \le m$,
$f_i(\mtx{X}^{\star}_{\tau_k}) \le 0$ and $f_i$ is lower
semi-continuous, $\mtx{X}_c$ obeys
\begin{equation}\label{eqn:constaintXc2}
f_i(\mtx{X}_c)\le 0, \quad i=1,\ldots,m.
\end{equation}
Furthermore, since $\|\mtx{X}_{\tau}^{\star}\|_F^2$ is bounded,
\eqref{eqn:XmustarX0} yields
$$
\limsup_{\tau\to\infty}\|\mtx{X}_{\tau}^{\star}\|_*\leq\|\mtx{X}_{\infty}\|_*,
\quad
\|\mtx{X}_{\infty}\|_*\leq\liminf_{\tau\to\infty}\|\mtx{X}_{\tau}^{\star}\|_*.
$$
An immediate consequence is
$\lim_{\tau\to\infty}\|\mtx{X}_{\tau}^{\star}\|_*=\|\mtx{X}_{\infty}\|_*$
and, therefore, $\|\mtx{X}_{c}\|_*=\|\mtx{X}_{\infty}\|_*$. This shows
that $\mtx{X}_c$ is a solution to \eqref{eqn:min}. Now it follows from
the definition of $\mtx{X}_\infty$ that $\|\mtx{X}_c\|_F \geq
\|\mtx{X}_{\infty}\|_F$, while we also have $\|\mtx{X}_c\|_F \leq
\|\mtx{X}_{\infty}\|_F$ because of \eqref{eqn:boundedXmustar}.  We
conclude that $\|\mtx{X}_c\|_F =\|\mtx{X}_{\infty}\|_F$ and thus
$\mtx{X}_c=\mtx{X}_{\infty}$ since $\mtx{X}_{\infty}$ is unique.
\end{proof}

\section{Convergence Analysis}
\label{sec:conv}

This section establishes the convergence of the SVT iterations.  We
begin with the simpler proof of the convergence of \eqref{eqn:iter} in
the special case of the matrix completion problem, and then present
the argument for the more general constraints
\eqref{eqn:itergeneral2}. We hope that this progression will make the
second and more general proof more transparent.

\subsection{Convergence for matrix completion}

We begin by recording a lemma which establishes the strong convexity
of the objective $f_\tau$.
\begin{lemma}
  \label{lem:alpha}
  Let $\mtx{Z} \in \partial f_\tau(\mtx{X})$ and $\mtx{Z}' \in \partial
  f_\tau(\mtx{X}')$. Then
\begin{equation}
\label{eq:alpha}
\<\mtx{Z} - \mtx{Z}', \mtx{X} - \mtx{X}'\> \ge \|\mtx{X} - \mtx{X}'\|_F^2.
\end{equation}
\end{lemma}
\begin{proof}
  An element $\mtx{Z}$ of $\partial f_\tau(\mtx{X})$ is of the form
  $\mtx{Z} = \tau \mtx{Z}_0 + \mtx{X}$, where $\mtx{Z}_0 \in \partial
  \|\mtx{X}\|_*$, and similarly for $\mtx{Z}'$. This gives
\[
\<\mtx{Z} - \mtx{Z}', \mtx{X} - \mtx{X}'\> = \tau \, \<\mtx{Z}_0 - \mtx{Z}_0', \mtx{X} - \mtx{X}'\> + \|\mtx{X}-\mtx{X}'\|_F^2
\]
and it thus suffices to show that the first term of the right-hand
side is nonnegative. From \eqref{eqn:subdiffNorm}, we have that any
subgradient of the nuclear norm at $\mtx{X}$ obeys $\|\mtx{Z}_0\|_2
\le 1$ and $\<\mtx{Z}_0, \mtx{X}\> = \|\mtx{X}\|_*$. In particular,
this gives
\begin{equation*}
|\<\mtx{Z}_0, \mtx{X}'\>|  \le \|\mtx{Z}_0\|_2
\|\mtx{X}'\|_* \le \|\mtx{X}'\|_*, \qquad
|\<\mtx{Z}'_0, \mtx{X}\>|  \le \|\mtx{Z}'_0\|_2
\|\mtx{X}\|_* \le \|\mtx{X}\|_*.
\end{equation*}
Whence,
\begin{align*}
 \<\mtx{Z}_0 - \mtx{Z}_0', \mtx{X} - \mtx{X}'\> & =  \<\mtx{Z}_0 , \mtx{X} \>  +  \<\mtx{Z}_0',\mtx{X}'\>  - \<\mtx{Z}_0, \mtx{X}'\>  -  \<\mtx{Z}_0', \mtx{X}\>\cr
& =    \|\mtx{X}\|_*  +  \|\mtx{X}'\|_*   - \<\mtx{Z}_0, \mtx{X}'\>  -  \<\mtx{Z}_0', \mtx{X}\> \ge 0,
\end{align*}
which proves the lemma.
\end{proof}

This lemma is key in showing that the SVT algorithm \eqref{eqn:iter} converges.
\begin{theorem}
\label{thm:converge} Suppose that the sequence of step sizes obeys
$0 < \inf \delta_k \le \sup \delta_k < 2$. Then the sequence
$\{\mtx{X}^k\}$ obtained via \eqref{eqn:iter} converges to the
unique solution of \eqref{eqn:minnuc+fro}.
\end{theorem}
\begin{proof}
  Let $(\mtx{X}^\star,\mtx{Y}^\star)$ be primal-dual optimal for the
  problem \eqref{eqn:minnuc+fro}. 
  The optimality conditions give
  \begin{align*}
    \mymathbf{0} & = \mtx{Z}^k - \mathcal{P}_{\Omega}(\mtx{Y}^{k-1})\cr
    \mymathbf{0} & = \mtx{Z}^\star -  \mathcal{P}_{\Omega}(\mtx{Y}^\star),
  \end{align*}
  for some $\mtx{Z}^k \in \partial f_\tau(\mtx{X}^k)$ and some
  $\mtx{Z}^\star \in \partial f_\tau(\mtx{X}^\star)$. We then deduce that
\[
(\mtx{Z}^k - \mtx{Z}^\star) - \mathcal{P}_{\Omega}(\mtx{Y}^{k-1} - \mtx{Y}^\star) = \mymathbf{0}
\]
and, therefore, it follows from Lemma \ref{lem:alpha} that
\begin{equation}
\label{eq:crucial2}
\<\mtx{X}^k - \mtx{X}^\star, \mathcal{P}_{\Omega}(\mtx{Y}^{k-1} - \mtx{Y}^\star)\> =
\<\mtx{Z}^k - \mtx{Z}^\star, \mtx{X}^k - \mtx{X}^\star\> \ge \|\mtx{X}^k - \mtx{X}^\star\|_F^2.
\end{equation}
We continue and observe that because $\cP_{\Omega}\mtx{X}^\star = \cP_{\Omega}\mtx{M}$,
\[
\|\mathcal{P}_{\Omega}(\mtx{Y}^{k} - \mtx{Y}^\star)\|_F  =
\|\mathcal{P}_{\Omega}(\mtx{Y}^{k-1} - \mtx{Y}^\star) + \delta_k \cP_{\Omega}(\mtx{X}^\star - \mtx{X}^k)\|_F.
\]
Therefore, setting $r_k = \|\mathcal{P}_{\Omega}(\mtx{Y}^{k} -
\mtx{Y}^\star)\|_F$,
\begin{align}
\label{eq:stepsize}
  r_k^2 & = r_{k-1}^2 - 2\delta_k \<\cP_{\Omega}(\mtx{Y}^{k-1} - \mtx{Y}^\star),\mtx{X}^k -
  \mtx{X}^\star\> + \delta_k^2 \|\cP_{\Omega}(\mtx{X}^\star - \mtx{X}^k)\|_F^2\cr & \le r_{k-1}^2 -
  2\delta_k \|\mtx{X}^k - \mtx{X}^\star\|_F^2 + \delta_k^2 \|\mtx{X}^k -
  \mtx{X}^\star\|_F^2
\end{align}
since for any matrix $\mtx{X}$, $\|\cP_\Omega(\mtx{X})\|_F \le
\|\mtx{X}\|_F$.  Under our assumptions about the size of $\delta_k$,
we have $2\delta_k - \delta_k^2 \ge \beta$ for all $k \ge 1$ and some
$\beta > 0$ and thus
\begin{equation}
\label{eq:crucial}
r_k^2 \le r_{k-1}^2 - \beta \|\mtx{X}^k - \mtx{X}^\star\|_F^2.
\end{equation}
Two properties follow from this:
\begin{enumerate}
\item The sequence $\{\|\mathcal{P}_{\Omega}(\mtx{Y}^{k} -
  \mtx{Y}^\star)\|_F\}$ is nonincreasing and, therefore, converges to
  a limit.
\item As a consequence, $\|\mtx{X}^k - \mtx{X}^\star\|_F^2 \goto 0$ as
  $k \goto \infty$.
\end{enumerate}
The theorem is established.
\end{proof}

\subsection{General convergence theorem}

Our second result is more general and establishes the convergence of
the SVT iterations to the solution of \eqref{eq:convex} under general
convex constraints. From now now, we will only assume that the
function $\cF(\mtx{X})$ is Lipschitz in the sense that
\begin{equation}
  \label{eq:Lipschitz}
  \|\cF(\mtx{X}) - \cF(\mtx{Y}\| \le L(\cF) \|\mtx{X} - \mtx{Y}\|_F,
\end{equation}
for some nonnegative constant $L(\cF)$.  Note that if $\cF$ is affine,
$\cF(\mtx{X}) = \vct{b} - \cA(\mtx{X})$, we have $L(\cF) = \|\cA\|_2$
where $\|\cA\|_2$ is the spectrum norm of the linear transformation
$\cA$ defined as $\|\cA\|_2:=\sup\{\|\cA(\mtx{X})\|_{\ell_2}:\|\mtx{X}\|_F=1\}$.
We also recall that $\cF(\mtx{X}) =
(f_1(\mtx{X}), \ldots, f_m(\mtx{X}))$ where each $f_i$ is convex, and
that the Lagrangian for the problem \eqref{eq:convex} is given by
\[
\cL(\mtx{X},\vct{y}) = f_\tau(\mtx{X}) + \<\vct{y}, \cF(\mtx{X})\>, \quad \vct{y} \ge \vct{0}.
\]
We will assume to simplify that strong duality holds which is
automatically true if the constraints obey constraint qualifications
such as Slater's condition \cite{BoydBook}.

We first establish the following preparatory lemma.
\begin{lemma}
  \label{teo:proj}
  Let $(\mtx{X}^\star,\vct{y}^\star)$ be a primal-dual optimal pair for
  \eqref{eq:convex}. Then for each $\delta > 0$, $\vct{y}^\star$ obeys
  \begin{equation}
    \label{eq:proj}
    \vct{y}^\star = [\vct{y}^\star + \delta \cF(\mtx{X}^\star)]_+.
  \end{equation}
\end{lemma}
\begin{proof}
  Recall that the projection $\vct{x}_0$ of a point $\vct{x}$ onto a convex
  set $\mathcal{C}$ is characterized by
\[
\begin{cases} \vct{x}_0 \in \mathcal{C},\cr
\<\vct{y}-\vct{x}_0, \vct{x} - \vct{x}_0\> \le 0, \,\, \forall \vct{y} \in \mathcal{C}.
\end{cases}
\]
In the case where $\mathcal{C} = \R^m_+ = \{\vct{x} \in \R^m : \vct{x} \ge \vct{0}\}$, this
condition becomes $\vct{x}_0 \ge \vct{0}$ and
\[
\<\vct{y} - \vct{x}_0, \vct{x} - \vct{x}_0\> \le 0, \,\,
\forall \vct{y} \ge \vct{0}.
\]

Now because $\vct{y}^\star$ is dual optimal we have
\[
\cL(\mtx{X}^\star,\vct{y}^\star) \ge \cL(\mtx{X}^\star,\vct{y}), \quad \forall \vct{y} \ge \vct{0}.
\]
Substituting the expression for the Lagrangian, this is equivalent to
\[
\<\vct{y} - \vct{y}^\star, \cF(\mtx{X}^\star)\> \le 0, \quad \forall \vct{y} \ge \vct{0},
\]
which is the same as
\[
\<\vct{y} - \vct{y}^\star, \vct{y}^\star + \rho \cF(\mtx{X}^\star) - \vct{y}^\star\> \le 0, \quad \forall \vct{y} \ge \vct{0}, \,\, \forall \rho \ge 0.
\]
Hence it follows that $\vct{y}^\star$ must be the projection of $\vct{y}^\star +
\rho \cF(\mtx{X}^\star)$ onto  the nonnegative orthant $\R^m_+$. Since
the projection of an arbitrary vector $\vct{x}$ onto $\R^m_+$ is given by
$\vct{x}_+$, our claim follows.
\end{proof}

We are now in the position to state our general convergence result.
\begin{theorem}
\label{thm:converge3}
Suppose that the sequence of step sizes obeys $0 < \inf \delta_k \le
\sup \delta_k < 2/\|L(\mathcal{F})\|^2$, where $L(\cF)$ is the
Lipschitz constant in \eqref{eq:Lipschitz}. Then assuming strong
duality, the sequence $\{\mtx{X}^k\}$ obtained via
\eqref{eqn:itergeneral2} converges to the unique solution of
\eqref{eq:convex}.
\end{theorem}
\begin{proof}
  Let $(\mtx{X}^\star,\vct{y}^\star)$ be primal-dual optimal for the
  problem \eqref{eq:convex}. We claim that the optimality conditions
  give that for all $\mtx{X}$
  \begin{align}
    \label{eq:subtle}
    \<\mtx{Z}^k, \mtx{X} - \mtx{X}^k\>  + \<\vct{y}^{k-1}, \cF(\mtx{X}) - \cF(\mtx{X}^k)\> & \ge 0,\cr
    \<\mtx{Z}^\star, \mtx{X} - \mtx{X}^\star\> +\<\vct{y}^\star, \cF(\mtx{X}) - \cF(\mtx{X}^\star)\> & \ge 0,
  \end{align}
  for some $\mtx{Z}^k \in \partial f_\tau(\mtx{X}^k)$ and some
  $\mtx{Z}^\star \in \partial f_\tau(\mtx{X}^\star)$. We justify this
  assertion by proving one of the two inequalities since the other is
  exactly similar. For the first, $\mtx{X}^k$ minimizes
  $\cL(\mtx{X},\vct{y}^{k-1})$ over all $\mtx{X}$ and, therefore,
  there exist $\mtx{Z}^k \in \partial f_\tau(\mtx{X}^k)$ and $\mtx{Z}_i^k
  \in \partial f_i(\mtx{X}^k)$, $1 \le i \le m$, such that
\[
  \mtx{Z}^k + \sum_{i  =  1}^m y_i^{k-1} \mtx{Z}_i^k = 0.
\]
Now because each $f_i$ is convex,
\[
f_i(\mtx{X}) - f_i(\mtx{X}^k) \ge \< \mtx{Z}_i^k, \mtx{X} - \mtx{X}^k\>
\]
and, therefore,
\[
\<\mtx{Z}^k, \mtx{X} - \mtx{X}^k\> + \sum_{i = 1}^m y_i^{k-1}
(f_i(\mtx{X}) - f_i(\mtx{X}^k)) \ge \< \mtx{Z}^k + \sum_{i = 1}^m
y_i^{k-1} \mtx{Z}_i^k, \mtx{X} - \mtx{X}^k\> = 0.
\]
This is \eqref{eq:subtle}.

Now write the first inequality in \eqref{eq:subtle} for
$\mtx{X}^\star$, the second for $\mtx{X}^k$ and sum the two
inequalities. This gives
\[
\<\mtx{Z}^k - \mtx{Z}^\star, \mtx{X}^k - \mtx{X}^\star\> + \<\vct{y}^{k-1} - \vct{y}^\star, \cF(\mtx{X}^k) -
\cF(\mtx{X}^\star)\> \le 0.
\]
The rest of the proof is essentially the same as that of Theorem
\ref{thm:converge2}. It follows from Lemma \ref{lem:alpha} that
\begin{equation}
\label{eq:crucial3}
\<\vct{y}^{k-1} - \vct{y}^\star, \cF(\mtx{X}^k) - \cF(\mtx{X}^\star)\> \le
-\<\mtx{Z}^k - \mtx{Z}^\star, \mtx{X}^k - \mtx{X}^\star\> \le -\|\mtx{X}^k - \mtx{X}^\star\|_F^2.
\end{equation}
We continue and observe that because $\vct{y}^\star = [\vct{y}^\star + \delta_k
\cF(\mtx{X})]_+$ by Lemma \ref{teo:proj}, we have
\begin{align*}
\|\vct{y}^{k} - \vct{y}^\star\| & = \|[\vct{y}^{k-1} + \delta_k \cF(\mtx{X}^k)]_+
- [\vct{y}^\star + \delta_k \cF(\mtx{X}^\star)]_+\|\cr
& \le \|\vct{y}^{k-1} - \vct{y}^\star + \delta_k (\cF(\mtx{X}^k)
- \cF(\mtx{X}^\star))\|
\end{align*}
since the projection onto the convex set $\R^m_+$ is a contraction.
Therefore,
\begin{align*}
\|\vct{y}^{k} - \vct{y}^\star\|^2 & = \|\vct{y}^{k-1} -
\vct{y}^\star\|^2 + 2\delta_k \, \<\vct{y}^{k-1} - \vct{y}^\star,
\cF(\mtx{X}^k) - \cF(\mtx{X}^\star)\> + \delta_k^2 \|\cF(\mtx{X}^k)
- \cF(\mtx{X}^\star)\|^2\cr & \le \|\vct{y}^{k-1} - \vct{y}^\star\|^2 -
2\delta_k \|\mtx{X}^k - \mtx{X}^\star\|_F^2 + \delta_k^2 L^2\,
\|\mtx{X}^k - \mtx{X}^\star\|_F^2,
\end{align*}
where we have put $L$ instead of $L(\cF)$ for short.  Under our
assumptions about the size of $\delta_k$, we have $2\delta_k -
\delta_k^2 L^2\ge \beta$ for all $k \ge 1$ and some $\beta > 0$. Then
\begin{equation}
\label{eq:crucialno}
\|\vct{y}^{k} - \vct{y}^\star\|^2 \le \|\vct{y}^{k-1} - \vct{y}^\star\|^2 - \beta \|\mtx{X}^k - \mtx{X}^\star\|_F^2,
\end{equation}
and the conclusion is as before.
\end{proof}

The problem \eqref{eq:linear} with linear constraints can be reduced
to \eqref{eq:convex} by choosing
\[
\cF(\mtx{X})=\left[\begin{matrix}\vct{b}\cr
    -\vct{b}\end{matrix}\right]- \left[\begin{matrix}\cA\cr
    -\cA\end{matrix}\right]\mtx{X},
\]
and we have the following corollary:
\begin{corollary}
  \label{thm:converge2} Suppose that the sequence of step sizes obeys
  $0 < \inf \delta_k \le \sup \delta_k < 2/\|\cA\|_2^2$. Then the
  sequence $\{\mtx{X}^k\}$ obtained via \eqref{eqn:itergeneral}
  converges to the unique solution of \eqref{eq:linear}.
\end{corollary}

Let $\|\cA\|_2 := \sup \{\|\cA(\mtx{X})\|_{F} : \|\mtx{X}\|_F = 1\}$.
With $\cF(\mtx{X})$ given as above, we have $|L(\cF)|^2 =
2\|\cA\|^2_2$ and thus, Theorem \ref{thm:converge3} guarantees
convergence as long as $0 < \inf \delta_k \le \sup \delta_k <
1/\|\cA\|_2^2$. However, an argument identical to the proof of Theorem
\ref{thm:converge} would remove the extra factor of two. We omit the
details.

\section{Implementation and Numerical Results}
\label{sec:num}

This section provides implementation details of the SVT algorithm---as
to make it practically effective for matrix completion---such as the
numerical evaluation of the singular value thresholding operator, the
selection of the step size $\delta_k$, the selection of a stopping
criterion, and so on. This section also introduces several numerical
simulation results which demonstrate the performance and effectiveness
of the SVT algorithm. We show that $30,000 \times 30,000$ matrices of
rank 10 are recovered from just about 0.4\% of their sampled entries
in a matter of a few minutes on a modest desktop computer with a 1.86
GHz CPU (dual core with Matlab's multithreading option enabled) and 3
GB of memory.

\subsection{Implementation details}
\label{sec:implementation}

\subsubsection{Evaluation of the singular value thresholding operator}

To apply the singular value tresholding operator at level $\tau$ to an
input matrix, it suffices to know those singular values and
corresponding singular vectors above the threshold $\tau$. In the
matrix completion problem, the singular value thresholding operator is
applied to sparse matrices $\{\mtx{Y}^k\}$ since the number of sampled
entries is typically much lower than the number of entries in the
unknown matrix $\mtx{M}$, and we are hence interested in numerical
methods for computing the dominant singular values and singular
vectors of large sparse matrices. The development of such methods is a
relatively mature area in scientific computing and numerical linear
algebra in particular. In fact, many high-quality packages are readily
available. Our implementation uses PROPACK, see \cite{Lar:Propack} for
documentation and availability. One reason for this choice is
convenience: PROPACK comes in a Matlab and a Fortran version, and we
find it convenient to use the well-documented Matlab version. More
importantly, PROPACK uses the iterative Lanczos algorithm to compute
the singular values and singular vectors directly, by using the
Lanczos bidiagonalization algorithm with partial
reorthogonalization. In particular, PROPACK does not compute the
eigenvalues and eigenvectors of $(\mtx{Y}^k)^*\mtx{Y}^k$ and
$\mtx{Y}^k(\mtx{Y}^k)^*$, or of an augmented matrix as in the Matlab
built-in function `\texttt{svds}' for example. Consequently, PROPACK
is an efficient---both in terms of number of flops and storage
requirement---and stable package for computing the dominant singular
values and singular vectors of a large sparse matrix. For information,
the available documentation \cite{Lar:Propack} reports a speedup
factor of about ten over Matlab's `\texttt{svds}'. Furthermore, the
Fortran version of PROPACK is about 3--4 times faster than the Matlab
version. Despite this significant speedup, we have only used the
Matlab version but since the singular value shrinkage operator is
by-and-large the dominant cost in the SVT algorithm, we expect that a
Fortran implementation would run about 3 to 4 times faster.

As for most SVD packages, though one can specify the number of
singular values to compute, PROPACK can not automatically compute only
those singular values exceeding the threshold $\tau$.  One must
instead specify the number $s$ of singular values ahead of time, and
the software will compute the $s$ largest singular values and
corresponding singular vectors. To use this package, we must then
determine the number $s_k$ of singular values of $\mtx{Y}^{k-1}$ to be
computed at the $k$th iteration. We use the following simple
method. Let $r_{k-1}=\rank(\mtx{X}^{k-1})$ be the number of nonzero
singular values of $\mtx{X}^{k-1}$ at the previous iteration. Set $s_k
= r_{k-1}+1$ and compute the first $s_{k}$ singular values of
$\mtx{Y}^{k-1}$. If some of the computed singular values are already
smaller than $\tau$, then $s_k$ is a right choice. Otherwise,
increment $s_k$ by a predefined integer $\ell$ repeatedly until some
of the singular values fall below $\tau$.  In the experiments, we
choose $\ell=5$.  Another rule might be to repeatedly multiply $s_k$
by a positive number---e.g.~2---until our criterion is
met. Incrementing $s_k$ by a fixed integer works very well in
practice; in our experiments, we very rarely need more than one
update.

We note that it is not necessary to rerun the Lanczos iterations for
the first $s_k$ vectors since they have been already computed; only a
few new singular values ($\ell$ of them) need to be numerically
evaluated. This can be done by modifying the PROPACK routines. We
have not yet modified PROPACK, however. Had we done so, our run times
would be decreased.

\subsubsection{Step sizes}

There is a large literature on ways of selecting a step size but for
simplicity, we shall use step sizes that are independent of the
iteration count; that is $\delta_k = \delta$ for $k = 1, 2, \ldots$.
From Theorem \ref{thm:converge}, convergence for the completion
problem is guaranteed \eqref{eqn:iter} provided that $0 < \delta <
2$. This choice is, however, too conservative and the convergence is
typically slow. In our experiments, we use instead
\begin{equation}
  \label{eq:heuristic}
  \delta  = 1.2 \, \frac{n_1 n_2}{m},
\end{equation}
i.e.~$1.2$ times the undersampling ratio. We give a heuristic
justification below.

Consider a fixed matrix $\mtx{A} \in \R^{n_1 \times n_2}$.  Under the
assumption that the column and row spaces of $\mtx{A}$ are not well
aligned with the vectors taken from the canonical basis of $\R^{n_1}$
and $\R^{n_2}$ respectively---the {\em incoherence assumption} in
\cite{CR:XXX:08}---then with very large probability over the choices
of $\Omega$, we have
\begin{equation}
  \label{eq:weakRIP}
  (1-\epsilon)  p \,  \|\mtx{A}\|_F^2  \le  \|\mathcal{P}_\Omega(\mtx{A})\|_F^2 \le   (1+\epsilon)  p \,  \|\mtx{A}\|_F^2, \quad p := m/(n_1n_2),
\end{equation}
provided that the rank of $\mtx{A}$ is not too large.  The probability
model is that $\Omega$ is a set of sampled entries of cardinality $m$
sampled uniformly at random so that all the choices are equally
likely. In \eqref{eq:weakRIP}, we want to think of $\epsilon$ as a
small constant, e.g.~smaller than 1/2. In other words, the `energy' of
$\mtx{A}$ on $\Omega$ (the set of sampled entries) is just about
proportional to the size of $\Omega$. The near isometry
\eqref{eq:weakRIP} is a consequence of Theorem 4.1 in
\cite{CR:XXX:08}, and we omit the details.

Now returning to the proof of Theorem \ref{thm:converge}, we see that a sufficient
condition for the convergence of \eqref{eqn:iter} is
\[
\exists\beta>0,\quad -2\delta \|\mtx{X}^\star - \mtx{X}^k\|_F^2 +
\delta^2 \|\cP_{\Omega}(\mtx{X}^\star -
\mtx{X}^k)\|_F^2\leq-\beta\|\mtx{X}^\star - \mtx{X}^k\|_F^2,
\]
compare \eqref{eq:crucial}, which is equivalent to
\[
0<\delta<2\frac{\|\mtx{X}^\star - \mtx{X}^k\|_F^2}
{\|\cP_{\Omega}(\mtx{X}^\star - \mtx{X}^k)\|_F^2}.
\]
Since $\|\cP_\Omega(\mtx{X})\|_F \le \|\mtx{X}\|_F$ for any matrix
$\mtx{X} \in \R^{n_1 \times n_2}$, it is safe to select $\delta <
2$. But suppose that we could apply \eqref{eq:weakRIP} to the matrix
$\mtx{A} = \mtx{X}^\star - \mtx{X}^k$. Then we could take $\delta$
inversely proportional to $p$; e.g.~with $\epsilon = 1/4$, we could
take $\delta \le 1.6 p^{-1}$. Below, we shall use the value $\delta =
1.2 p^{-1}$ which allows us to take large steps and still provides
convergence, at least empirically.

The reason why this is not a rigorous argument is that
\eqref{eq:weakRIP} cannot be applied to $\mtx{A} = \mtx{X}^\star -
\mtx{X}^k$ even though this matrix difference may obey the incoherence
assumption.  The issue here is that $\mtx{X}^\star - \mtx{X}^k$ is not
a fixed matrix, but rather depends on $\Omega$ since the iterates
$\{\mtx{X}^k\}$ are computed with the knowledge of the sampled set.

\subsubsection{Initial steps}

The SVT algorithm starts with $\mtx{Y}^0=\mtx{0}$, and we want to
choose a large $\tau$ to make sure that the solution of
\eqref{eqn:minnuc+fro} is close enough to a solution of
\eqref{eqn:min}. Define $k_0$ as that integer obeying
\begin{equation}\label{eqn:k0}
\frac{\tau}{\delta\|\mathcal{P}_{\Omega}(\mtx{M})\|_2} \in (k_0-1,
k_0].
\end{equation}
Since $\mtx{Y}^0=\mtx{0}$, it is not difficult to see that
\[
\mtx{X}^k = \mtx{0}, \quad \mtx{Y}^k= k\delta\,
\mathcal{P}_{\Omega}(\mtx{M}), \quad k = 1, \ldots, k_0.
\]
To save work, we may simply skip the computations of
$\mtx{X}^1,\ldots,\mtx{X}^{k_0}$, and start the iteration by computing
$\mtx{X}^{k_0+1}$ from $\mtx{Y}^{k_0}$.

This strategy is a special case of a {\em kicking device} introduced
in \cite{ODY:XXX:08}; the main idea of such a kicking scheme is that
one can `jump over' a few steps whenever possible. Just like in the
aforementioned reference, we can develop similar kicking strategies
here as well.  Because in our numerical experiments the kicking is
rarely triggered, we forgo the description of such strategies.

\subsubsection{Stopping criteria}

Here, we discuss stopping criteria for the sequence of SVT iterations
\eqref{eqn:iter}, and present two possibilities.

The first is motivated by the first-order optimality conditions or KKT
conditions tailored to the minimization problem
\eqref{eqn:minnuc+fro}. By \eqref{eqn:argminequiv} and letting
$\partial_{\mtx{Y}} g_0(\mtx{Y})=\mtx{0}$ in \eqref{eqn:partialg0}, we
see that the solution $\mtx{X}^\star_{\tau}$ to \eqref{eqn:minnuc+fro}
must also verify
\begin{equation}\label{eqn:KKT}
\begin{cases}
\mtx{X}=\mathcal{D}_{\tau}(\mtx{Y}),\cr
\mathcal{P}_{\Omega}(\mtx{X}-\mtx{M})=\mtx{0},
\end{cases}
\end{equation}
where $\mtx{Y}$ is a matrix vanishing outside of $\Omega^c$.
Therefore, to make sure that $\mtx{X}^k$ is close to
$\mtx{X}^{\star}_\tau$, it is sufficient to check how close
$(\mtx{X}^k,\mtx{Y}^{k-1})$ is to obeying \eqref{eqn:KKT}. By
definition, the first equation in \eqref{eqn:KKT} is always
true. Therefore, it is natural to stop \eqref{eqn:iter} when the error
in the second equation is below a specified tolerance. We suggest
stopping the algorithm when
\begin{equation}\label{eqn:stop0}
\frac{\|\mathcal{P}_{\Omega}(\mtx{X}^k-\mtx{M})\|_F}{\|\mathcal{P}_{\Omega}(\mtx{M})\|_F}\leq\epsilon,
\end{equation}
where $\epsilon$ is a fixed tolerance, e.g.~$10^{-4}$. We provide a
short heuristic argument justifying this choice below.

In the matrix completion problem, we know that under suitable
assumptions
\[
\|\mathcal{P}_\Omega(\mtx{M})\|_F^2 \asymp p \, \|\mtx{M}\|_F^2,
\]
which is just \eqref{eq:weakRIP} applied to the fixed matrix $\mtx{M}$
(the symbol $\asymp$ here means that there is a constant $\epsilon$ as
in \eqref{eq:weakRIP}). Suppose we could also apply \eqref{eq:weakRIP}
to the matrix $\mtx{X}^k - \mtx{M}$ (which we rigorously cannot since
$\mtx{X}^k$ depends on $\Omega$), then we would have
\begin{equation}\label{eq:xk-m}
\|\mathcal{P}_\Omega(\mtx{X}^k - \mtx{M})\|_F^2 \asymp p \,
\|\mtx{X}^k - \mtx{M}\|_F^2,
\end{equation}
and thus
\[
\frac{\|\mathcal{P}_{\Omega}(\mtx{X}^k-\mtx{M})\|_F}{\|\mathcal{P}_{\Omega}(\mtx{M})\|_F}\asymp
\frac{\|\mtx{X}^k - \mtx{M}\|_F}{\|\mtx{M}\|_F}.
\]
In words, one would control the relative reconstruction error by
controlling the relative error on the set of sampled locations.

A second stopping criterion comes from duality theory. Firstly, the
iterates $\mtx{X}^k$ are generally not feasible for
\eqref{eqn:minnuc+fro} although they become asymptotically
feasible. One can construct a feasible point from $\mtx{X}^k$ by
projecting it onto the affine space $\{\mtx{X} :
\mathcal{P}_{\Omega}(\mtx{X}) = \mathcal{P}_{\Omega}(\mtx{M})\}$ as
follows:
\[
\tilde{\mtx{X}}^k = \mtx{X}^k +  \mathcal{P}_{\Omega}(\mtx{M}-\mtx{X}^k).
\]
As usual let $f_\tau(\mtx{X}) = \tau \|\mtx{X}\|_* +\frac{1}{2}
\|\mtx{X}\|_F^2$ and denote by $p^\star$ the optimal value of
\eqref{eqn:minnuc+fro}. Since $\tilde{\mtx{X}}^k$ is feasible, we
have
\[
p^\star \le f_\tau(\tilde{\mtx{X}}^k) := b_k.
\]
Secondly, using the notations of Section \ref{sec:uzawa},
duality theory gives that
\[
a_k := g_0(\mtx{Y}^{k-1}) = {\cal L}(\mtx{X}^k,\mtx{Y}^{k-1}) \le p^\star.
\]
Therefore, $b_k - a_k$ is an upper bound on the duality gap and one
can stop the algorithm when this quantity falls below a given
tolerance.

For very large problems in which one holds $\mtx{X}^k$ in reduced SVD
form, one may not want to compute the projection $\tilde{\mtx{X}}^k$
since this matrix would not have low rank and would require
significant storage space (presumably, one would not want to spend
much time computing this projection either). Hence, the second method
only makes practical sense when the dimensions are not prohibitively
large, or when the iterates do not have low rank.

\subsubsection{Algorithm}

We conclude this section by summarizing the implementation details and
give the SVT algorithm for matrix completion below (Algorithm
\ref{alg:SVT}). Of course, one would obtain  a very similar
structure for the more general problems of the form \eqref{eq:linear}
and \eqref{eq:convex} with linear inequality constraints.  For
convenience, define for each nonnegative integer $s \le
\min\{n_1,n_2\}$,
\[
[\mtx{U}^k,\mtx{\Sigma}^k,\mtx{V}^k]_s, \quad k = 1, 2, \ldots,
\]
where $\mtx{U}^{k} = [\vct{u}_1^{k}, \ldots, \vct{u}_{s}^{k}]$ and  $\mtx{V}^{k} =
[\vct{v}_1^{k}, \ldots, \vct{v}_s^{k}]$ are the first $s$ singular vectors
of the matrix $\mtx{Y}^{k}$, and $\mtx{\Sigma}^{k}$ is a diagonal
matrix with the first $s$ singular values $\sigma^{k}_1, \ldots,
\sigma_s^{k}$ on the diagonal.

\begin{algorithm}[htb]
\caption{Singular Value Thresholding (SVT) Algorithm}{}
\label{alg:SVT} \centering \fbox{
\begin{minipage}{.9\textwidth}
  \vspace{4pt} \alginout{sampled set $\Omega$ and sampled entries
    $\mathcal{P}_{\Omega}(\mtx{M})$, step size $\delta$, tolerance
    $\epsilon$, parameter $\tau$, increment $\ell$, and maximum
    iteration count $k_{\max}$}  {$\mtx{X}^{\mathrm{opt}}$}
  \algdescript{Recover a low-rank matrix $\mtx{M}$ from a subset of
    sampled entries}  \vspace{6pt}

\begin{algtab}
  Set $\mtx{Y}^0 = k_0\delta \, \mathcal{P}_{\Omega}(\mtx{M})$ ($k_0$ is defined in \eqref{eqn:k0}) \\
  Set $r_0=0$\\
  \algforto{$k=1$}{$k_{\max}$}
  Set $s_k=r_{k-1}+1$\\
  \algrepeat
  Compute $[\mtx{U}^{k-1},\mtx{\Sigma}^{k-1},\mtx{V}^{k-1}]_{s_k}$\\
  Set $s_k = s_k + \ell$\\
 \alguntil{$\sigma_{s_k-\ell}^{k-1}\le\tau$}
 Set $r_{k} = \max\{j : \sigma^{k-1}_j > \tau\}$\\
 Set $\mtx{X}^k = \sum_{j = 1}^{r_k}
  (\sigma_j^{k-1}-\tau) \vct{u}_j^{k-1} \vct{v}^{k-1}_j$\\

 \algifthen{$\|\mathcal{P}_{\Omega}(\mtx{X}^k-\mtx{M})\|_F/\|\mathcal{P}_{\Omega}\mtx{M}\|_F
           \leq\epsilon$}{{\bf break}}

  Set
    $Y^{k}_{ij}=\begin{cases} 0&\mbox{if }(i,j)\not\in\Omega,\cr
      Y^{k-1}_{ij}+\delta(M_{ij}-X_{ij}^k)&\mbox{if }(i,j)\in\Omega \end{cases}
    $ \\
 \algend {\bf end} {\em  for $k$} \\
 Set $\mtx{X}^{\mathrm{opt}}= \mtx{X}^k$\\
 \algend
\end{algtab}
\end{minipage}}
\end{algorithm}


\subsection{Numerical results}
\label{sec:results}

\subsubsection{Linear equality constraints} 

Our implementation is in Matlab and all the computational results we
are about to report were obtained on a desktop computer with a 1.86
GHz CPU (dual core with Matlab's multithreading option enabled) and 3
GB of memory. In our simulations, we generate $n \times n$ matrices of
rank $r$ by sampling two $n\times r$ factors $\mtx{M}_L$ and
$\mtx{M}_R$ independently, each having i.i.d.~Gaussian entries, and
setting $\mtx{M}=\mtx{M}_L\mtx{M}_R^*$ as it is suggested in
\cite{CR:XXX:08}. The set of observed entries $\Omega$ is sampled
uniformly at random among all sets of cardinality $m$.

The recovery is
performed via the SVT algorithm (Algorithm \ref{alg:SVT}), and we use
\begin{equation}\label{eqn:stop}
\|\mathcal{P}_{\Omega}(\mtx{X}^k-\mtx{M})\|_F/
\|\mathcal{P}_{\Omega}\mtx{M}\|_F<10^{-4}
\end{equation}
as a stopping criterion.  As discussed earlier, the step sizes are
constant and we set $\delta=1.2 p^{-1}$.  Throughout this section, we
denote the output of the SVT algorithm by $\mtx{X}^{\mathrm{opt}}$.
The parameter $\tau$ is chosen empirically and set to $\tau = 5n$. A
heuristic argument is as follows. Clearly, we would like the term
$\tau \|\mtx{M}\|_*$ to dominate the other, namely,
$\frac{1}{2}\|\mtx{M}\|_F^2$. For products of Gaussian matrices as
above, standard random matrix theory asserts that the Frobenius norm
of $\mtx{M}$ concentrates around $n\sqrt{r}$, and that the nuclear
norm concentrates around about $nr$ (this should be clear in the
simple case where $r = 1$ and is generally valid). The value $\tau=5n$
makes sure that on the average, the value of $\tau \|\mtx{M}\|_*$ is
about $10$ times that of $\frac12\|\mtx{M}\|_F^2$ as long as the rank
is bounded away from the dimension $n$.

\begin{table}
\begin{center}
\begin{tabular}{cccc|ccc}\hline
  \multicolumn{4}{c|}{Unknown $\mtx{M}$}&\multicolumn{3}{c}{Computational results}\\ \hline
  size ($n\times n$)& rank ($r$) & $m/d_r$ & $m/n^2$ & time(s) & \# iters & relative error\\ \hline
  & 10 & 6 & 0.12 & 23 & 117 & $1.64\times10^{-4}$ \\
  $1,000\times 1,000$ & 50 & 4 & 0.39 & 196 & 114 & $1.59\times10^{-4}$ \\
  & 100 & 3 & 0.57 & 501 & 129 & $1.68\times10^{-4}$ \\ \hline

  & 10 & 6 & 0.024 & 147 & 123 & $1.73\times10^{-4}$ \\
 $5,000\times 5,000$ & 50 & 5 & 0.10 & 950 & 108 & $1.61\times10^{-4}$ \\
  & 100 & 4 & 0.158 & 3,339 & 123 & $1.72\times10^{-4}$ \\ \hline

  & 10 & 6 & 0.012 & 281 & 123 & $1.73\times10^{-4}$ \\
 $10,000\times 10,000$ & 50 & 5 & 0.050 & 2,096 & 110 & $1.65\times10^{-4}$ \\
  & 100 & 4 & 0.080 & 7,059 & 127 & $1.79\times10^{-4}$ \\ \hline

  & 10 & 6 & 0.006 & 588 & 124 & $1.73\times10^{-4}$ \\
 \raisebox{1.5ex}[0pt]{$20,000\times 20,000$} & 50 & 5 & 0.025 & 4,581 & 111 & $1.66\times10^{-4}$ \\ \hline

 $30,000\times30,000$ & 10 & 6 & 0.004 & 1,030 & 125 & $1.73\times10^{-4}$ \\ \hline
\end{tabular}
\end{center}
\caption{Experimental results for matrix
  completion. The rank $r$ is the rank of the unknown matrix
  $\mtx{M}$, $m/d_r$ is the ratio between the number of sampled entries and the number of degrees of freedom in an $n \times n$ matrix of rank $r$
  (oversampling ratio), and $m/n^2$ is the fraction of observed entries.  All the computational results on the right are averaged over five runs.}
\label{tab:result1}
\end{table}

Our computational results are displayed in Table
\ref{tab:result1}. There, we report the run time in seconds, the
number of iterations it takes to reach convergence \eqref{eqn:stop},
and the relative error of the reconstruction
\begin{equation}\label{eqn:relerr}
  \mathrm{relative~
    error}=\|\mtx{X}^{\mathrm{opt}}-\mtx{M}\|_F/\|\mtx{M}\|_F,
\end{equation}
where $\mtx{M}$ is the real unknown matrix.  All of these quantities
are averaged over five runs. The table also gives the percentage of
entries that are observed, namely, $m/n^2$ together with a quantity
that we may want to think as the information oversampling ratio.
Recall that an $n \times n$ matrix of rank $r$ depends upon $d_r :=
r(2n-r)$ degrees of freedom. Then $m/d_r$ is the ratio between the
number of sampled entries and the `true dimensionality' of an $n \times n$
matrix of rank $r$.

The first observation is that the SVT algorithm performs extremely
well in these experiments. In all of our experiments, it takes fewer
than 200 SVT iterations to reach convergence. As a consequence, the
run times are short. As indicated in the table, we note that one
recovers a $1,000\times 1,000$ matrix of rank $10$ in less than a
minute. The algorithm also recovers $30,000 \times 30,000$ matrices of
rank $10$ from about $0.4\%$ of their sampled entries in just about 17
minutes.  In addition, higher-rank matrices are also efficiently
completed: for example, it takes between one and two hours to recover
$10,000 \times 10,000$ matrices of rank $100$ and $20,000\times
20,000$ matrices of rank $50$. We would like to stress that these
numbers were obtained on a modest CPU (1.86GHz). Furthermore, a
Fortran implementation is likely to cut down on these numbers by a
multiplicative factor typically between three and four.

We also check the validity of the stopping criterion \eqref{eqn:stop}
by inspecting the relative error defined in \eqref{eqn:relerr}. The
table shows that the heuristic and nonrigorous analysis of Section
\ref{sec:implementation} holds in practice since the relative
reconstruction error is of the same order as
$\|\mathcal{P}_{\Omega}(\mtx{X}^{\mathrm{opt}}-\mtx{M})\|_F/
\|\mathcal{P}_{\Omega}\mtx{M}\|_F \sim 10^{-4}$. Indeed, the overall
relative errors reported in Table \ref{tab:result1} are all less than
$2\times10^{-4}$.

We emphasized all along an important feature of the SVT algorithm,
which is that the matrices $\mtx{X}^k$ have low rank.  We demonstrate
this fact empirically in Figure \ref{fig:rank}, which plots the rank
of $\mtx{X}^k$ versus the iteration count $k$, and does this for
unknown matrices of size $5,000\times 5,000$ with different ranks. The
plots reveal an interesting phenomenon: in our experiments, the rank
of $\mtx{X}^{k}$ is nondecreasing so that the maximum rank is reached
in the final steps of the algorithm. In fact, the rank of the iterates
quickly reaches the value $r$ of the true rank. After these few
initial steps, the SVT iterations search for that matrix with rank $r$
minimizing the objective functional.  As mentioned earlier, the
low-rank property is crucial for making the algorithm run fast.
\begin{figure}[h]
  \begin{center}
    \begin{tabular}{ccc}
\includegraphics[width=.32\textwidth]{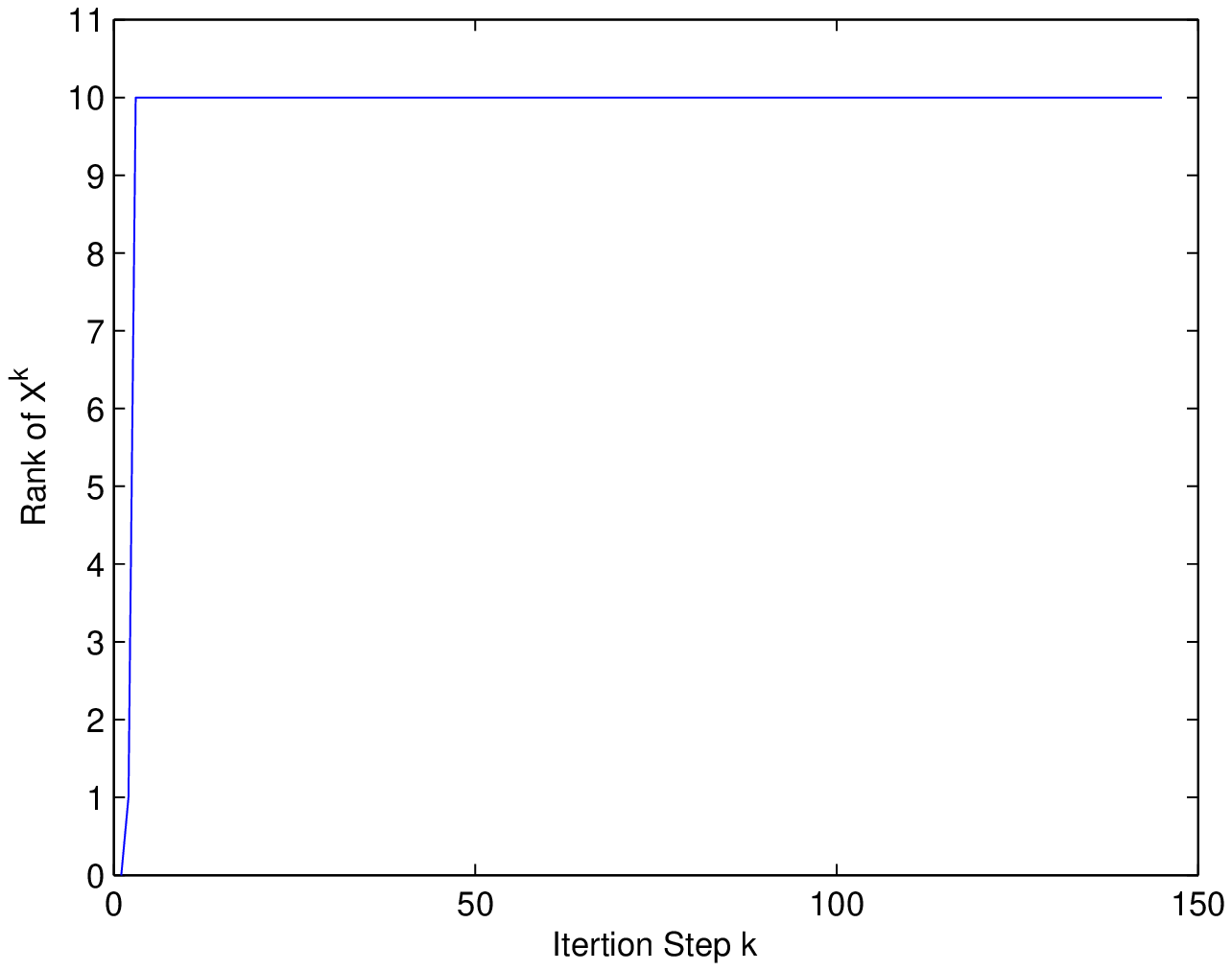} &
\includegraphics[width=.32\textwidth]{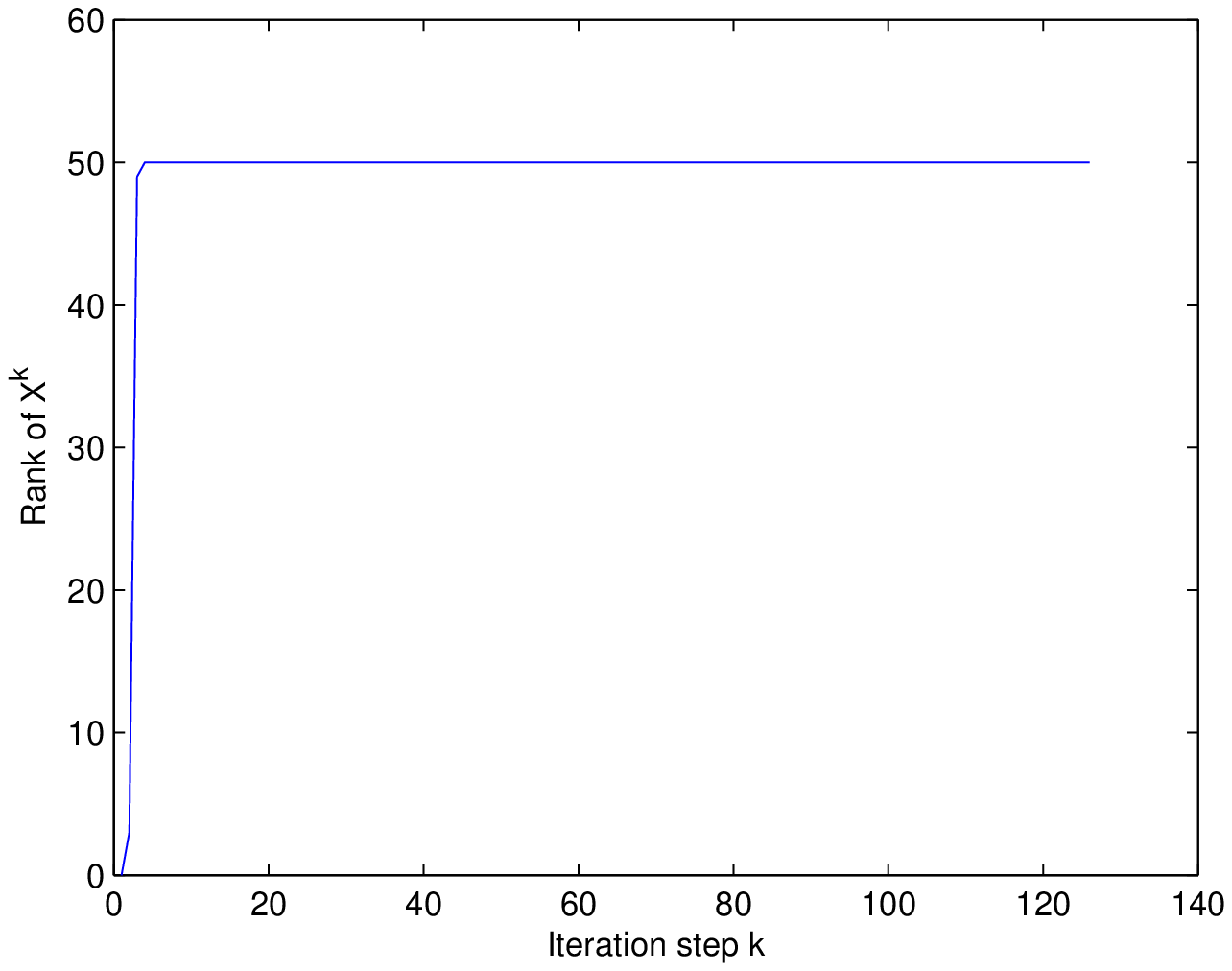} &
\includegraphics[width=.32\textwidth]{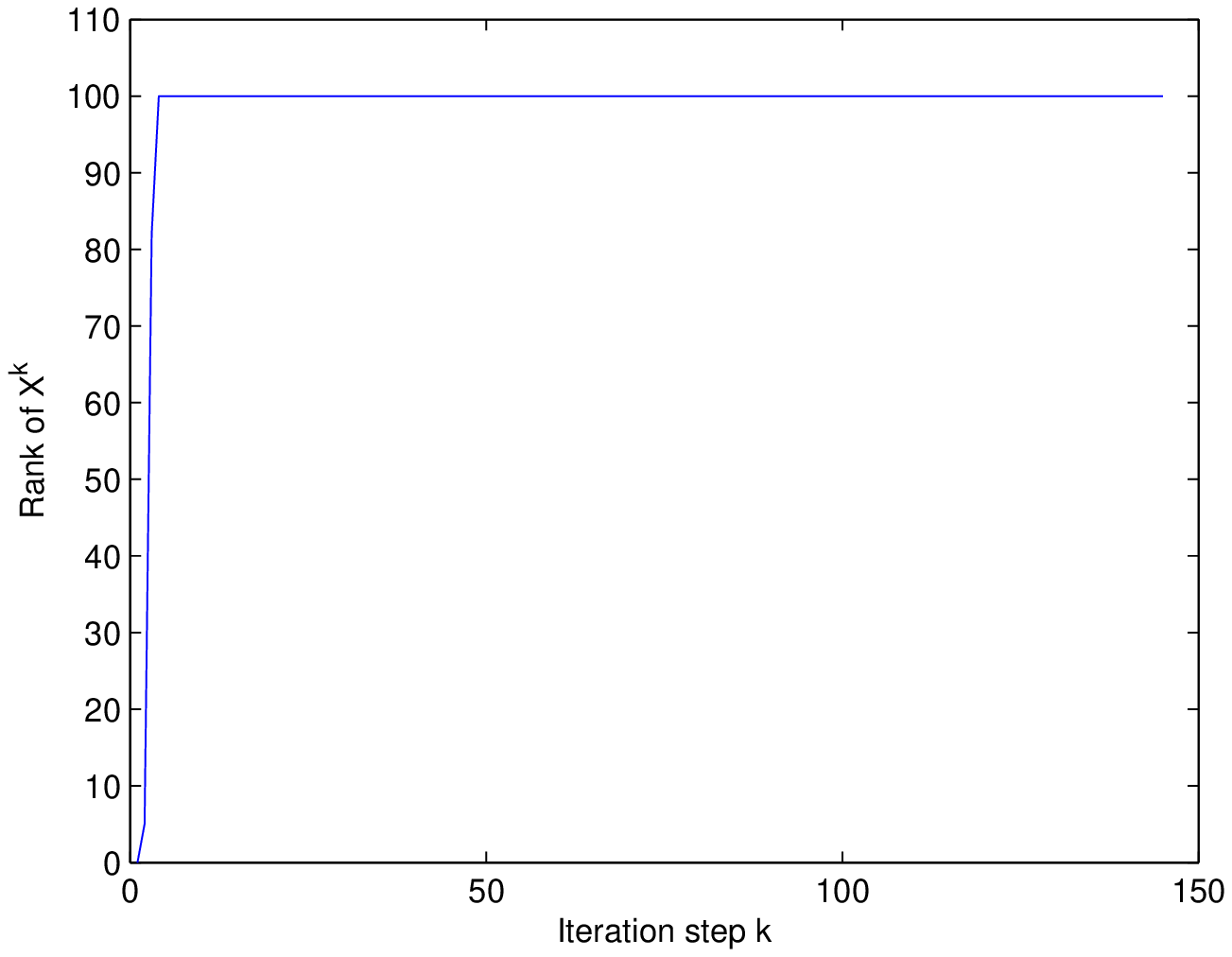} \\
$r = 10$ & $r = 50$ & $r = 100$
\end{tabular}
\end{center}
\caption{Rank of $\mtx{X}^k$ as a function $k$ when the unknown matrix
  $\mtx{M}$ is of size $5,000 \times 5,000$ and of rank $r$.}
\label{fig:rank}
\end{figure}

Finally, we demonstrate the results of the SVT algorithm for matrix
completion from noisy sampled entries. Suppose we observe data from
the model
\begin{equation}
\label{eq:noisy2}
{B}_{ij} = M_{ij} + Z_{ij}, \qquad (i,j) \in \Omega,
\end{equation}
where $\mtx{Z}$ is a zero-mean Gaussian white noise with standard
deviation $\sigma$. We run the SVT algorithm but stop early, as soon
as $\mtx{X}^k$ is consistent with the data and obeys
\begin{equation}\label{eqn:stop1}
\|\mathcal{P}_{\Omega}(\mtx{X}^{k}-\mtx{B})\|_F^2
 \leq (1+\epsilon) \, m\sigma^2,
 \end{equation}
 where $\epsilon$ is a small parameter.  Our reconstruction
 $\hat{\mtx{M}}$ is the first $\mtx{X}^k$ obeying \eqref{eqn:stop1}.
 The results are shown in Table \ref{tab:noise} (the quantities are
 averages of 5 runs). Define the noise ratio as
\[
\|\cP_\Omega(\mtx{Z})\|_F/\|\mathcal{P}_{\Omega}(\mtx{M})\|_F,
\]
and the relative error by \eqref{eqn:relerr}. From Table
\ref{tab:noise}, we see that the SVT algorithm works well as the
relative error between the recovered and the true data matrix is just
about equal to the noise ratio.
\begin{table}
\begin{center}
\begin{tabular}{c|cccc|ccc}\hline
  &\multicolumn{4}{c|}{Unknown matrix $\mtx{M}$}&\multicolumn{3}{c}{Computational results}\\
  \cline{2-8}
  \raisebox{2ex}[0pt]{\small{noise ratio}}
  & size ($n\times n$) & rank ($r$) & $m/d_r$ & $m/n^2$ & time(s)& \# iters & relative error\\\hline

 & & 10 & 6 & 0.12 & 10.8 & 51 & $0.78\times10^{-2}$ \\
 $10^{-2}$&$1,000\times 1,000$ & 50 & 4 & 0.39 & 87.7 & 48 & $0.95\times10^{-2}$ \\
 & & 100 & 3 & 0.57 & 216 & 50 & $1.13\times10^{-2}$ \\ \hline

 & & 10 & 6 & 0.12 & 4.0 & 19 & $0.72\times10^{-1}$ \\
 $10^{-1}$&$1,000\times 1,000$ & 50 & 4 & 0.39 & 33.2 & 17 & $0.89\times10^{-1}$ \\
 & & 100 & 3 & 0.57 & 85.2 & 17 & $1.01\times10^{-1}$ \\ \hline

 & & 10 & 6 & 0.12 & 0.9 & 3 & $0.52$ \\
 1&$1,000\times 1,000$ & 50 & 4 & 0.39 & 7.8 & 3 & $0.63$ \\
 & & 100 & 3 & 0.57 & 34.8 & 3 & $0.69$ \\ \hline
\end{tabular}
\end{center}
\caption{Simulation results for noisy data. The computational results are averaged over five runs.}
\label{tab:noise}
\end{table}

\newcommand{\sol}{\hat{\mtx{M}}}
\newcommand{\true}{\mtx{M}}

The theory of low-rank matrix recovery from noisy data is nonexistent
at the moment, and is obviously beyond the scope of this paper. Having
said this, we would like to conclude this section with an intuitive
and nonrigorous discussion, which may explain why the observed
recovery error is within the noise level.  Suppose again that
$\sol$ obeys \eqref{eq:xk-m}, namely,
 \begin{equation}\label{eqn:error}
 \|\cP_\Omega(\sol - \true)\|_F^2 \asymp p
\|\sol - \true\|_F^2.
\end{equation}
As mentioned earlier, one condition for this to happen is that $\true$
and $\sol$ have low rank. This is the reason why it is important to
stop the algorithm early as we hope to obtain a solution which is both
consistent with the data and has low rank (the limit of the SVT
iterations, $\lim_{k \goto \infty} \mtx{X}^k$, will not generally have
low rank since there may be no low-rank matrix matching the noisy
data). From
\[
\|\cP_\Omega(\sol - \true)\|_F \le \|\cP_\Omega(\sol-\mtx{B})\|_F +
\|\cP_\Omega({\mtx{B}}-\mtx{M})\|_F,
\]
and the fact that both terms on the right-hand side are on the order
of $\sqrt{m\sigma^2}$, we would have $p \|\sol - \true\|_F^2 = O(m
\sigma^2)$ by \eqref{eqn:error}. In particular, this would give that
the relative reconstruction error is on the order of the noise
ratio since $\|\cP_\Omega(\true)\|_F^2 \asymp p \|\true\|_F^2$---as
observed experimentally.

\subsubsection{Linear inequality constraints} 

We now examine the speed at which one can solve similar problems with
linear inequality constraints instead of linear equality
constraints. We assume the model \eqref{eq:noisy2}, where the matrix
$\true$ of rank $r$ is sampled as before, and solve the problem
\eqref{eq:DS} by using \eqref{eqn:iterDS2}. We formulate the
inequality constraints in \eqref{eq:DS} with $E_{ij} = \sigma$ so that
one searches for a solution $\sol$ with minimum nuclear norm among all
those matrices whose sampled entries deviate from the observed ones by
at most the noise level $\sigma$.\footnote{This may not be
  conservative enough from a statistical viewpoint but this works well
  in this case, and our emphasis here is on computational rather than
  statistical issues.}  In this experiment, we adjust $\sigma$ to be
one tenth of a typical absolute entry of $\mtx{M}$, i.e.~$\sigma = 0.1
\, \sum_{ij \in \Omega} |M_{ij}|/m$, and the noise ratio as defined
earlier is 0.780. We set $n = 1,000$, $r = 10$, and the number $m$ of
sampled entries is five times the number of degrees of freedom,
i.e.~$m = 5 d_r$. Just as before, we set $\tau = 5n$, and choose a
constant step size $\delta = 1.2p^{-1}$.

\begin{figure}[h]
  \begin{center}
    \begin{tabular}{cc}
\includegraphics[width=.38\textwidth]{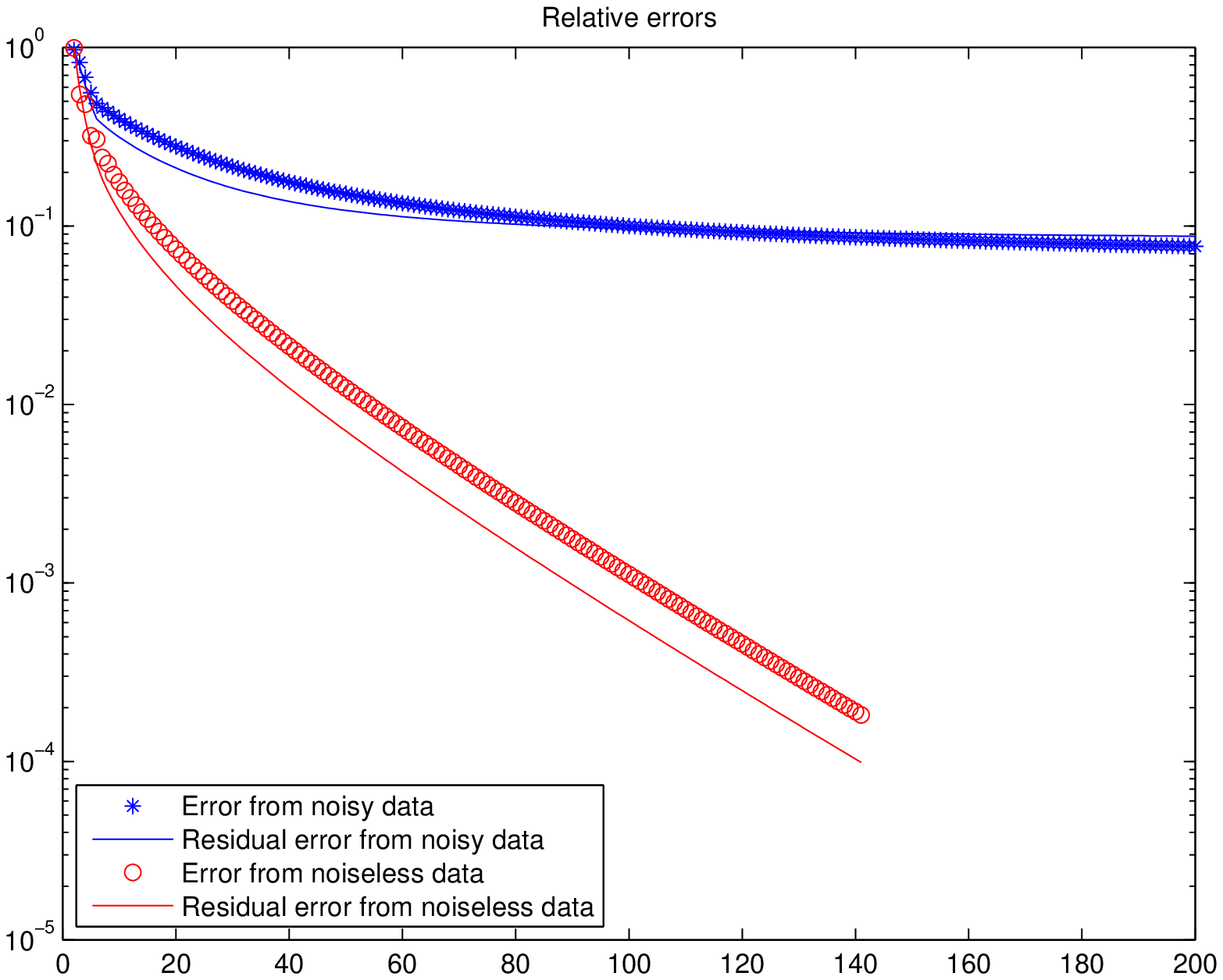} &
\includegraphics[width=.38\textwidth]{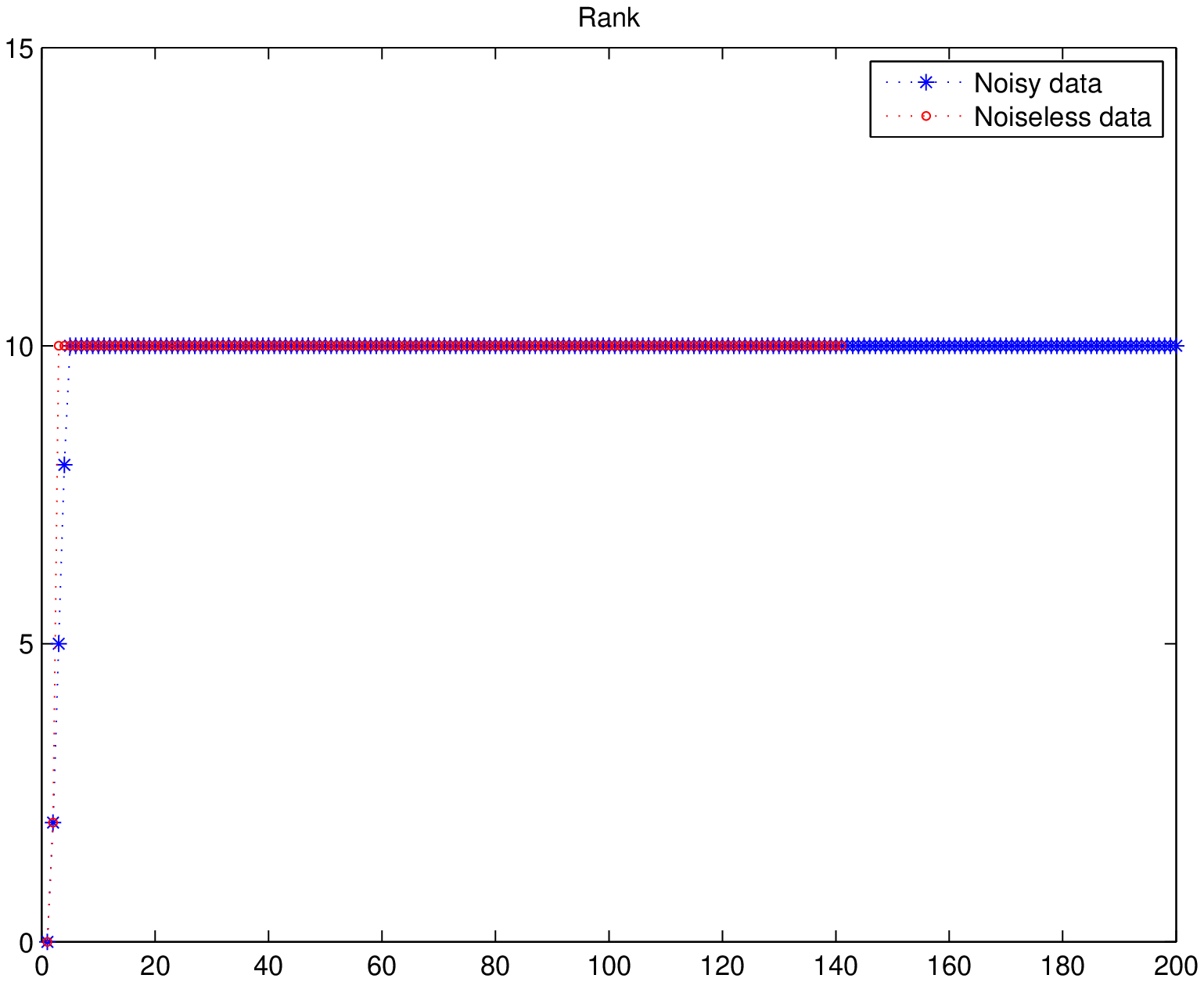} \\
(a) & (b)\\
\includegraphics[width=.38\textwidth]{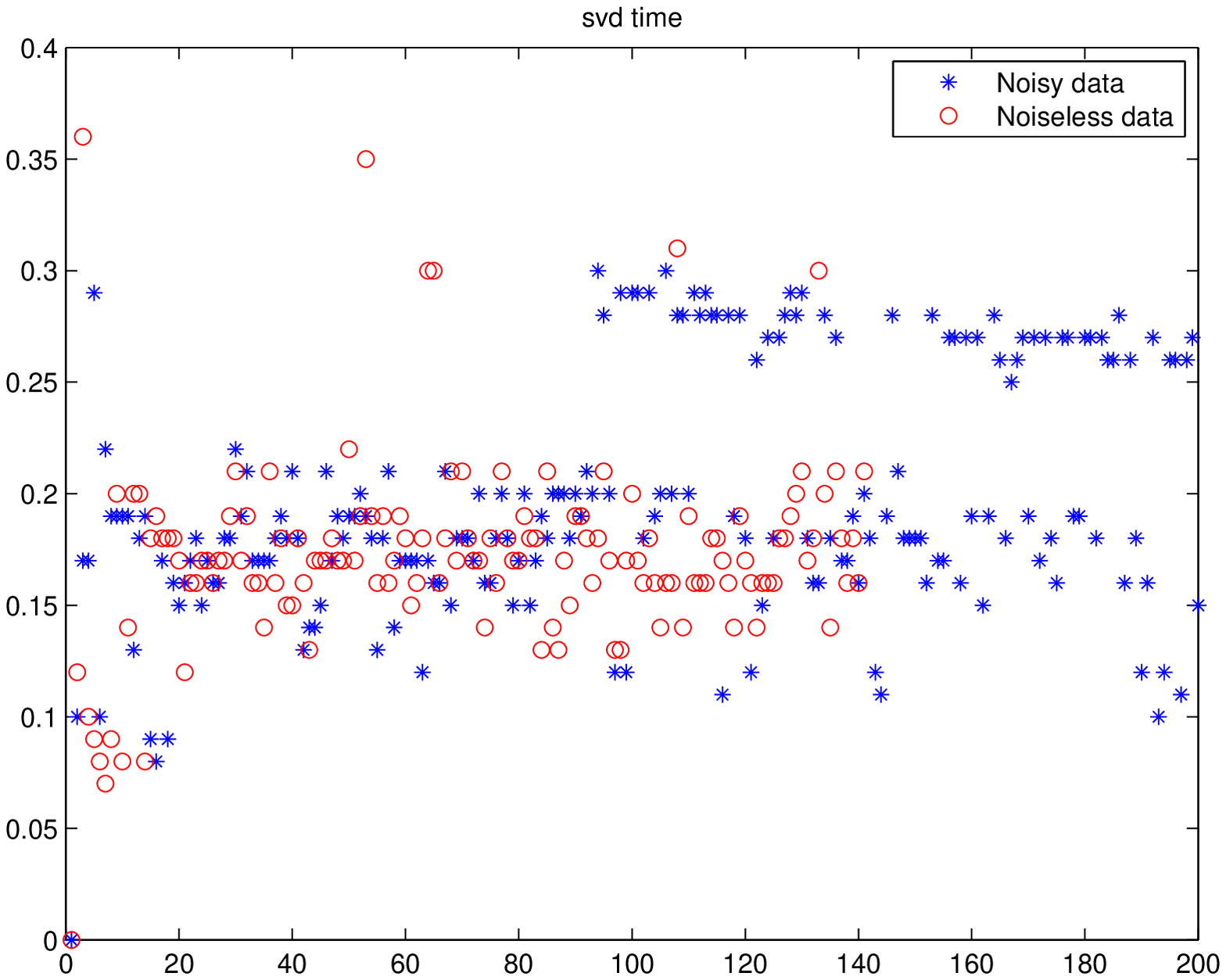} & \\
(c) & 
\end{tabular}
\end{center}
\caption{Computational results of the algorithm applied to noisy
  (linear inequality constraints as in \eqref{eq:DS}) and noiseless
  data (equality constraints). The blue (resp.~red) color is used for
  the noisy (resp.~noiseless) experiment.  (a) Plot of the
  reconstruction errors from noisy and noiseless data as a function of
  the iteration count. The thin line is the residual relative error
  $\|\cP_\Omega(\mtx{X^k}-\mtx{M})\|_F/\|\cP_\Omega(\mtx{M})\|_F$ and
  the thick line is the overall relative error
  $\|\mtx{X^k}-\mtx{M}\|_F/\|\mtx{M}\|_F$. (b) Rank of the iterates as
  a function of the iteration count. (c) Time it takes to compute the
  singular value thresholding operation as a function of the iteration
  count. The computer here is a single-core 3.00GHz Pentium 4 running
  Matlab 7.2.0.}
\label{fig:DS}
\end{figure}

The results, reported in Figure \ref{fig:DS}, show that the algorithm
behaves just as well with linear inequality constraints. To make this
point, we compare our results with those obtained from noiseless data
(same unknown matrix and sampled locations). In the noiseless case, it
takes about 150 iterations to reach the tolerance $\epsilon = 10^{-4}$
whereas in the noisy case, convergence occurs in about 200 iterations
(Figure \ref{fig:DS}(a)). In addition, just as in the noiseless
problem, the rank of the iterates is nondecreasing and quickly reaches
the true value $r$ of the rank of the unknown matrix $\mtx{M}$ we wish
to recover (Figure \ref{fig:DS}(b)). As a consequence the SVT
iterations take about the same amount of time as in the noiseless case
(Figure \ref{fig:DS}(c)) so that the total running time of the
algorithm does not appear to be substantially different from that in
the noiseless case.

We close by pointing out that from a statistical point of view, the
recovery of the matrix $\mtx{M}$ from undersampled and noisy entries
by the matrix equivalent of the Dantzig selector appears to be
accurate since the relative error obeys $\|\sol-\true\|_F/\|\true\|_F =
0.0769$ (recall that the noise ratio is about $0.08$).


\section{Discussion}
\label{sec:discussion}

This paper introduced a novel algorithm, namely, the singular value
thresholding algorithm for matrix completion and related nuclear norm
minimization problems.  This algorithm is easy to implement and
surprisingly effective both in terms of computational cost and
storage requirement when the minimum nuclear-norm solution is also
the lowest-rank solution.  We would like to close this paper by
discussing a few open problems and research directions related to this
work.

Our algorithm exploits the fact that the sequence of iterates
$\{\mtx{X}^k\}$ have low rank when the minimum nuclear solution has
low rank.  An interesting question is whether one can prove (or
disprove) that in a majority of the cases, this is indeed the case.

It would be interesting to explore other ways of computing
$\cD_\tau(\mtx{Y})$---in words, the action of the singular value shrinkage
operator. Our approach uses the Lanczos bidiagonalization algorithm
with partial reorthogonalization which takes advantages of sparse
inputs but other approaches are possible. We mention two of them.
\begin{enumerate}
\item A series of papers have proposed the use of randomized procedures
  for the approximation of a matrix $\mtx{Y}$ with a matrix $\mtx{Z}$ of
  rank $r$ \cite{RokhlinQR1,RokhlinQR2}. When this approximation
  consists of the truncated SVD retaining the part of the expansion
  corresponding to singular values greater than $\tau$, this can be
  used to evaluate $\cD_\tau(\mtx{Y})$. Some of these algorithms are
  efficient when the input $\mtx{Y}$ is sparse \cite{RokhlinQR1}, and
  it would be interesting to know whether these methods are fast and
  accurate enough to be used in the SVT iteration \eqref{eqn:iter}.

\item A wide range of iterative methods for computing matrix functions
  of the general form $f(\mtx{Y})$ are available today, see
  \cite{Higham} for a survey. A valuable research direction is to
  investigate whether some of these iterative methods, or other to be
  developed, would provide powerful ways for computing
  $\cD_\tau(\mtx{Y})$.
\end{enumerate}

In practice, one would like to solve \eqref{eqn:minnuc+fro} for large
values of $\tau$. However, a larger value of $\tau$ generally means a
slower rate of convergence. A good strategy might be to start with
a value of $\tau$, which is large enough so that \eqref{eqn:minnuc+fro}
admits a low-rank solution, and at the same time for which the
algorithm converges rapidly. One could then use a continuation method
as in \cite{Continuation} to increase the value of $\tau$ sequentially
according to a schedule $\tau_0, \tau_1, \ldots$, and use the solution
to the previous problem with $\tau = \tau_{i-1}$ as an initial guess
for the solution to the current problem with $\tau = \tau_i$ (warm
starting). We hope to report on this in a separate paper.

\small

\subsection*{Acknowledgments}
J-F.~C.~is supported by the Wavelets and Information Processing
Programme under a grant from DSTA, Singapore. E.~C.~is partially
supported by the Waterman Award from the National Science Foundation
and by an ONR grant N00014-08-1-0749. Z.~S.~is supported in part by
Grant R-146-000-113-112 from the National University of
Singapore. E.~C.~would like to thank Benjamin Recht and Joel Tropp for
fruitful conversations related to this project, and Stephen Becker for
his help in preparing the computational results of Section 5.2.2.

\bibliographystyle{abbrv}

\begin{thebibliography}{1}

\bibitem{Abernethy06}
J.~Abernethy, F.~Bach, T.~Evgeniou, and J.-P. Vert.
\newblock Low-rank matrix factorization with attributes.
\newblock Technical Report N24/06/MM, Ecole des Mines de Paris, 2006.

\bibitem{NetflixPrize}
ACM SIGKDD and Netflix.
\newblock {\em Proceedings of KDD Cup and Workshop}, 2007.
\newblock Proceedings available online at
  \url{http://www.cs.uic.edu/~liub/KDD-cup-2007/proceedings.html}.

\bibitem{Amit07}
Y.~Amit, M.~Fink, N.~Srebro, and S.~Ullman.
\newblock Uncovering shared structures in multiclass classification.
\newblock In {\em Proceedings of the Twenty-fourth International Conference on
  Machine Learning}, 2007.

\bibitem{Argyriou07}
A.~Argyriou, T.~Evgeniou, and M.~Pontil.
\newblock Multi-task feature learning.
\newblock In {\em Neural Information Processing Systems}, 2007.


\bibitem{BBAC:ECCV:04}
J.~Bect, L.~Blanc-F{\'e}raud, G.~Aubert, and A.~Chambolle,
\newblock A $\ell_1$ unified variational framework for image restoration,
\newblock in {\em Proc. Eighth Europ. Conf. Comput. Vision}, 2004.


\bibitem{BoydBook} S. Boyd, and L. Vandenberghe.
\newblock {\em Convex Optimization}.
\newblock Cambridge University Press, 2004.

\bibitem{CCSS:SISC:08}
J.-F. Cai, R.~Chan, L.~Shen, and Z.~Shen.
\newblock Restoration of chopped and nodded images by framelets.
\newblock {\em SIAM J. Sci. Comput.}, 30(3):1205--1227, 2008.

\bibitem{CCS:ACHA:08}
J.-F. Cai, R.~H. Chan, and Z.~Shen.
\newblock A framelet-based image inpainting algorithm.
\newblock {\em Appl. Comput. Harmon. Anal.}, 24(2):131--149, 2008.

\bibitem{COS:XXX:08:3}
J.-F. Cai, S.~Osher, and Z.~Shen.
\newblock {\em Convergence of the Linearized {B}regman Iteration for
  $\ell_1$-norm Minimization}, 2008.
\newblock UCLA CAM Report (08-52).

\bibitem{COS:XXX:08}
J.-F. Cai, S.~Osher, and Z.~Shen.
\newblock {\em Linearized {B}regman Iterations for Compressed Sensing}, 2008.
\newblock Math. Comp., to appear; see also UCLA CAM Report (08-06).

\bibitem{COS:XXX:08:2}
J.-F. Cai, S.~Osher, and Z.~Shen.
\newblock {\em Linearized {B}regman Iterations for Frame-Based Image
  Deblurring}, 2008.
\newblock preprint.


\bibitem{TVSynthesis}
E.~J. Cand\`es, and F.~Guo.
\newblock New multiscale transforms, minimum total variation synthesis:
  Applications to edge-preserving image reconstruction.
\newblock {\em Signal Processing}, 82:1519--1543, 2002.

\bibitem{CR:XXX:08}
E.~J. Cand{\`e}s and B.~Recht.
\newblock {\em Exact Matrix Completion via Convex Optimization}, 2008.

\bibitem{CR:IP:07}
E.~J. Cand{\`e}s and J.~Romberg.
\newblock Sparsity and incoherence in compressive sampling.
\newblock {\em Inverse Problems}, 23(3):969--985, 2007.

\bibitem{CRT:TIT:06}
E.~J. Cand{\`e}s, J.~Romberg, and T.~Tao.
\newblock Robust uncertainty principles: exact signal reconstruction from
  highly incomplete frequency information.
\newblock {\em IEEE Trans. Inform. Theory}, 52(2):489--509, 2006.

\bibitem{CT:TIT:05}
E.~J. Cand{\`e}s and T.~Tao.
\newblock Decoding by linear programming.
\newblock {\em IEEE Trans. Inform. Theory}, 51(12):4203--4215, 2005.

\bibitem{CT:TIT:06}
E.~J. Cand{\`e}s and T.~Tao.
\newblock Near-optimal signal recovery from random projections: universal
  encoding strategies?
\newblock {\em IEEE Trans. Inform. Theory}, 52(12):5406--5425, 2006.

\bibitem{DS} E.~J. Cand\`es and T.~Tao.
\newblock The Dantzig selector: statistical estimation when $p$ is
  much larger than $n$.
\newblock {\em Annals of Statistics} 35:2313--2351, 2007.

\bibitem{CS:NM:07}
A.~Chai and Z.~Shen.
\newblock Deconvolution: A wavelet frame approach.
\newblock {\em Numer. Math.}, 106(4):529--587, 2007.

\bibitem{CCSS:SISC:03}
R.~H. Chan, T.~F. Chan, L.~Shen, and Z.~Shen.
\newblock Wavelet algorithms for high-resolution image reconstruction.
\newblock {\em SIAM J. Sci. Comput.}, 24(4):1408--1432 (electronic), 2003.

\bibitem{ChenSuter} P. Chen, and D.~Suter.
\newblock Recovering the
  missing components in a large noisy low-rank matrix: application to
  {SFM} source.
\newblock {\em IEEE Transactions on Pattern Analysis
    and Machine Intelligence}, 26(8):1051-1063, 2004.

\bibitem{CW:MMS:05}
P.~L. Combettes and V.~R. Wajs.
\newblock Signal recovery by proximal forward-backward splitting.
\newblock {\em Multiscale Model. Simul.}, 4(4):1168--1200 (electronic), 2005.

\bibitem{DO:XXX:07}
J.~Darbon and S.~Osher.
\newblock {\em Fast discrete optimization for sparse approximations and
  deconvolutions}, 2007.
\newblock preprint.

\bibitem{DDD:CPAM:04}
I.~Daubechies, M.~Defrise, and C.~De~Mol.
\newblock An iterative thresholding algorithm for linear inverse problems with
  a sparsity constraint.
\newblock {\em Comm. Pure Appl. Math.}, 57(11):1413--1457, 2004.

\bibitem{DTV:IPI:07}
I.~Daubechies, G.~Teschke, and L.~Vese.
\newblock Iteratively solving linear inverse problems under general convex
  constraints.
\newblock {\em Inverse Probl. Imaging}, 1(1):29--46, 2007.


\bibitem{Don:TIT:06}
D.~L. Donoho.
\newblock Compressed sensing.
\newblock {\em IEEE Trans. Inform. Theory}, 52(4):1289--1306, 2006.

\bibitem{ESQD:ACHA:05}
M.~Elad, J.-L. Starck, P.~Querre, and D.~L. Donoho.
\newblock Simultaneous cartoon and texture image inpainting using morphological
  component analysis ({MCA}).
\newblock {\em Appl. Comput. Harmon. Anal.}, 19(3):340--358, 2005.


\bibitem{FSM:CJ:07}
M.~J.~Fadili, J.-L.~Starck, and F.~Murtagh.
\newblock Inpainting and zooming using sparse representations.
\newblock {\em The Computer Journal}, to appear.


\bibitem{FazelThesis}
M.~Fazel.
\newblock {\em Matrix Rank Minimization with Applications}.
\newblock PhD thesis, Stanford University, 2002.

\bibitem{fazelRank}
M.~Fazel, H.~Hindi, and S.~Boyd,
\newblock Log-det heuristic for matrix rank minimization with applications to
  {H}ankel and {E}uclidean distance matrices.
\newblock in {\em Proc. Am. Control Conf.}, June 2003.

\bibitem{Nowak_EM} M.~Figueiredo, and R.~Nowak,
\newblock An {EM} algorithm for wavelet-based image restoration.
\newblock {\em IEEE Transactions on Image Processing}, 12(8):906--916,
2003.

\bibitem{GO:XXX:08}
T.~Goldstein and S.~Osher.
\newblock {\em The Split {B}regman Algorithm for L1 Regularized Problems},
  2008.
\newblock UCLA CAM Reprots (08-29).


\bibitem{YinFPC} E.~T. Hale, W. Yin, and Y. Zhang.  \newblock {\em
    Fixed-point continuation for l1-minimization: methodology and
    convergence}.  \newblock 2008.  \newblock preprint.

\bibitem{Higham}
N.~J.~Higham.
\newblock {\em Functions of Matrices: {Theory} and Computation}.
\newblock Society for Industrial and Applied Mathematics,
Philadelphia, PA, USA, 2008.

\bibitem{HL:BOOK:93}
J.-B. Hiriart-Urruty and C.~Lemar{\'e}chal.
\newblock {\em Convex analysis and minimization algorithms. {I}}, volume 305 of
  {\em Grundlehren der Mathematischen Wissenschaften [Fundamental Principles of
  Mathematical Sciences]}.
\newblock Springer-Verlag, Berlin, 1993.
\newblock Fundamentals.

\bibitem{Lar:Propack}
R.~M. Larsen, \newblock {\em PROPACK -- Software for large and
sparse SVD calculations},
\newblock Available from
\url{http://sun.stanford.edu/~rmunk/PROPACK/}.

\bibitem{Lew:MP:03}
A.~S. Lewis.
\newblock The mathematics of eigenvalue optimization.
\newblock {\em Math. Program.}, 97(1-2, Ser. B):155--176, 2003.
\newblock ISMP, 2003 (Copenhagen).

\bibitem{RokhlinQR2}
E.~Liberty, F.~Woolfe, P.-G.~Martinsson, V.~Rokhlin, and M.~Tygert.
\newblock Randomized algorithms for the low-rank approximation of
            matrices.
\newblock {\em Proc. Natl. Acad. Sci. USA}, 104(51): 20167--20172, 2007.

\bibitem{Lintner}
S.~Lintner, and F. Malgouyres.
\newblock Solving a variational image restoration model
which involves $\ell_\infty$ constraints.
\newblock {\em Inverse Problems}, 20:815--831, 2004.

\bibitem{VandenbergheNuc}
Z. Liu, and L. Vandenberghe.
\newblock Interior-point method for nuclear norm approximation with
application to system identification.
\newblock submitted to {\em Mathematical Programming}, 2008.

\bibitem{RokhlinQR1}
P.-G.~Martinsson, V.~Rokhlin, and M.~Tygert.
\newblock A randomized algorithm for the approximation of matrices
\newblock Department of Computer Science, Yale University, New Haven,
CT, Technical Report 1361, 2006.

\bibitem{Mesbahi97}
M.~Mesbahi and G.~P. Papavassilopoulos.
\newblock On the rank minimization problem over a positive semidefinite linear
  matrix inequality.
\newblock {\em IEEE Transactions on Automatic Control}, 42(2):239--243, 1997.

\bibitem{OBGXY:MMS:05}
S.~Osher, M.~Burger, D.~Goldfarb, J.~Xu, and W.~Yin.
\newblock An iterative regularization method for total variation-based image
  restoration.
\newblock {\em Multiscale Model. Simul.}, 4(2):460--489 (electronic), 2005.

\bibitem{ODY:XXX:08}
S.~Osher, Y.~Mao, B.~Dong, and W.~Yin.
\newblock {\em Fast Linearized Bregman Iteration for Compressed Sensing and
  Sparse Denoising}, 2008.
\newblock UCLA CAM Reprots (08-37).

\bibitem{Recht07}
B.~Recht, M.~Fazel, and P.~Parrilo.
\newblock Guaranteed minimum rank solutions of matrix equations via nuclear
  norm minimization.
\newblock 2007.
\newblock Submitted to {\em SIAM Review}.


\bibitem{SDC:AA:03}
J.-L.~Starck, D.~L.~Donoho, and E.~J.~Cand{\`e}s,
\newblock Astronomical image representation by the curvelet
   transform.
\newblock {\em Astronom. and Astrophys.}, 398:785--800, 2003.

\bibitem{TTT:SDPT3}
K.~C. Toh, M.~J. Todd, and R.~H. T\"{u}t\"{u}nc\"{u}.
\newblock {\em SDPT3 -- a MATLAB software package for semidefinite-quadratic-linear
programming},
\newblock Available from
\url{http://www.math.nus.edu.sg/~mattohkc/sdpt3.html}.

\bibitem{Tomasi}
C. Tomasi and T. Kanade.
\newblock Shape and motion from image streams under orthography: a
  factorization method.
\newblock {\em International Journal of Computer Vision},
  9(2):137--154, 1992.

\bibitem{Wat:LAA:92}
G.~A. Watson.
\newblock Characterization of the subdifferential of some matrix norms.
\newblock {\em Linear Algebra Appl.}, 170:33--45, 1992.

\bibitem{Continuation}
S.~J.~Wright, R.~Nowak, and M.~Figueiredo.
\newblock Sparse reconstruction by separable approximation.
\newblock Submitted for publication, 2007.


\bibitem{YOGD:SIIMS:08}
W.~Yin, S.~Osher, D.~Goldfarb, and J.~Darbon.
\newblock {B}regman iterative algorithms for $\ell_1$-minimization with
  applications to compressed sensing.
\newblock {\em SIAM J. Imaging Sci.}, 1(1):143--168, 2008.



\end{thebibliography}

\end{document}